\documentclass[10pt]{article}
\usepackage{epsfig}
\usepackage{amssymb,amsmath,amsthm,url}
\usepackage{multirow}
\usepackage{booktabs}
\usepackage{natbib}
\usepackage{hyperref}
\usepackage{setspace}
\usepackage{dsfont}

\usepackage{amsfonts,times,eucal}
\usepackage{graphicx,ifthen}
\usepackage{wrapfig}
\usepackage{latexsym}
\usepackage{color}
\usepackage{mathrsfs}
\usepackage{bm}
\usepackage{bbm}
\usepackage{srctex}
\usepackage{enumitem}
\makeatletter
\newcommand{\mylabel}[2]{#2\def\@currentlabel{#2}\label{#1}}
\makeatother
\usepackage{mathtools}

\theoremstyle{plain}
\newtheorem{theorem}{Theorem}[section]
\newtheorem{corollary}[theorem]{Corollary}
\newtheorem{lemma}[theorem]{Lemma}

\newtheorem{condition}[theorem]{Condition}

\theoremstyle{definition}

\newtheorem{example}[theorem]{Example}

\theoremstyle{remark}

\numberwithin{equation}{section}

\newcommand{\N}{\mathbb{N}}
\newcommand{\R}{\mathbb{R}}

\newcommand{\Z}{\mathbb{Z}}

\newcommand{\p}{\mathbb{P}}

\newcommand{\cC}{\mathcal{C}}
\newcommand{\cD}{\mathcal{D}}
\newcommand{\cE}{\mathcal{E}}
\newcommand{\cG}{\mathcal{G}}
\newcommand{\cF}{\mathcal{F}}
\newcommand{\cH}{\mathcal{H}}
\newcommand{\cI}{\mathcal{I}}
\newcommand{\cL}{\mathcal{L}}
\newcommand{\cS}{\mathcal{S}}

\newcommand{\cO}{\mathcal{O}}

\newcommand{\fG}{\mathfrak{G}}

\newcommand{\E}[1]{\mathbb{E}\left [  #1 \right ]}
\newcommand{\Ec}[1]{\mathbb{E}^*\left [  #1 \right ]}

\newcommand{\veps}{\epsilon}
\let\veps\epsilon
\renewcommand{\epsilon}{\varepsilon}

\newcommand{\intd}[1]{\mathrm{d}#1}

\newcommand{\norm}[1]{\left\lVert #1 \right\rVert}

\newcommand{\scalar}[2]{\left\langle #1,#2 \right\rangle}

\newcommand{\1}[1]{\,\mathds{1}\! \left\{ #1 \right\} }

\newcommand{\beps}{\epsilon^{*}}
\newcommand{\bY}{Y^{*}}
\newcommand{\bootr}{\hat{r}^{*}_{n,b,h} }
\newcommand{\oas}{o_{a.s.}}
\newcommand{\Oas}{\cO_{a.s.}}

\allowdisplaybreaks

\begin{document}

\title{
The bootstrap in kernel regression for stationary ergodic data\\ when both response and predictor are functions\footnote{This research was supported by the German Research Foundation (DFG), Grant Number KR-4977/1.}
}

\author{Johannes T. N. Krebs\footnote{Department of Statistics, University of California, Davis, CA, 95616, USA, email: \tt{jtkrebs@ucdavis.edu} }\; \footnote{Corresponding author}
}

\date{\today}
\maketitle

\begin{abstract}
\setlength{\baselineskip}{1.8em}
We consider the double functional nonparametric regression model $Y=r(X)+\epsilon$, where the response variable $Y$ is Hilbert space-valued and the covariate $X$ takes values in a pseudometric space. The data satisfy an ergodicity criterion which dates back to \cite{laib2010nonparametric} and are arranged in a triangular array. So our model also applies to samples obtained from spatial processes, e.g., stationary random fields indexed by the regular lattice $\Z^N$ for some $N\in\N_+$.

We consider a kernel estimator of the Nadaraya--Watson type for the regression operator $r$ and study its limiting law which is a Gaussian operator on the Hilbert space. Moreover, we investigate both a naive and a wild bootstrap procedure in the double functional setting and demonstrate their asymptotic validity. This is quite useful as building confidence sets based on an asymptotic Gaussian distribution is often difficult. \medskip

\end{abstract}

	\vspace{-3em}
	\section*{}
	The seminal work of \cite{bosq_linear_2000} has initiated a lot of research on the theory and applications of functional data analysis. In this manuscript, we study aspects of the kernel regression model introduced by \cite{ferraty2000dimension}. More precisely, we consider a double functional regression problem and a corresponding version of the double functional wild bootstrap for a triangular array of dependent pairs $( (X_{n,i},Y_{n,i}) : 1\le i\le n, n\in\N)$ with identical distributions in the framework of stationary ergodic data of \cite{laib2010nonparametric}. The response variables $Y_{n,i}$ are Hilbert space-valued, the predictor variables $X_{n,i}$ take values in a functional space $\cE$ equipped with a pseudometric $d$.

	There is an extensive literature on asymptotic properties of nonparametric methods for functional data, in particular, the functional regression problem has been studied in several ways. A general introduction to functional data and their applications also offer the monographs of \cite{ramsay_applied_2002, ramsay_functional_2005}.

	Kernel methods for functional data are studied in the monograph of \cite{ferraty2006nonparametric}. \cite{ferraty2004nonparametric} consider the kernel regression model for i.i.d.\ data and a real-valued response variable. \cite{masry2005nonparametric} and \cite{delsol_advances_2009} extend the study to dependent data. \cite{ferraty_regression_2012} study the same model for a functional response and independent data. \cite{laib2011rates} obtain uniform rates for the kernel estimator.

	A functional version of the wild bootstrap for a real-valued response variable in the kernel regression model was proposed by \cite{ferraty_nonparametric_2007}; its asymptotic properties are studied in \cite{ferraty_validity_2010}. \cite{rana_bootstrap_2016} and \cite{politis_kernel_2016} give generalizations to strong mixing processes. \cite{gonzalez2011bootstrap} study the naive and the wild bootstrap in a linear regression model for FPCA-type estimates.

	In this paper, we denote the orthonormal basis of the separable Hilbert space $\cH$ by $\{e_k:k\in\N\}$ and its inner product by $\scalar{\cdot}{\cdot}$. Let $v\in\cH$ and define for a generic observation $(X,Y)\in \cE\times \cH$ the conditional mean functions by
	\begin{align*}
			r(x) &\coloneqq \E{ Y | X=x} \text{ and } r_v(x) \coloneqq  \E{	\scalar{Y}{v} | X=x }.
	\end{align*}
	Note that $r$ takes values in $\cH$ while $r_v$ is real-valued. We aim at estimating the regression operator $r$ with a Nadaraya--Watson-type estimator $\hat{r}_{n,h}$. 

	As mentioned above this problem has already been studied in the case where the data is $\alpha$-mixing. It is known that $\alpha$-mixing and stationarity, imply ergodicity, see, e.g., \cite{hannan2009multiple} and \cite{tempelman2010ergodic}. However, $\alpha$-mixing is sometimes too restrictive as it requires certain smoothness conditions. A well-known counterexample is the AR(1)-process 
	$	X_t =   \frac{1}{2} X_{t-1} + \epsilon_t$, $t\in\Z$
	with Bernoulli innovations $\epsilon_t$, see \cite{andrews1984non}. This process fails to be $\alpha$-mixing.	Consequently, alternative dependence concepts are also relevant when studying functional data such as functional time series or functional spatial processes. 
	
	In this paper, we continue with the concept of functional stationary ergodic data introduced in \cite{laib2010nonparametric}. We study the asymptotic normality of the kernel estimator in the double functional setting and prove the consistency of the wild and the naive bootstrap approximation of this kernel estimate in the Hilbert space. Therefore, we write $F^{\cF_{n,i-1} }_x(h) = \p( d(x,X_{n,i})\le h | \cF_{n,i-1} )$ for the conditional distribution function of $d(x,X_{n,i})$ given the past $\cF_{n,i-1}$ and $F_x(h) = \p( d(x,X_{n,i})\le h)$ for the unconditional distribution function. The unique feature of the underlying framework is the assumption on the ergodicity of the averages
	$
			n^{-1} \sum_{i=1}^n F^{\cF_{n,i-1} }_x(h) \approx F_x(h) 
	$
	in a sense which will be clarified below. Based on this assumption and on a multiplicative structure of the (conditional) small ball probabilities for $h$ tending to zero, we can deduce convergence results and limiting laws of the bootstrap in the double functional setting.

	The contribution of this paper is to provide advances when both the response and the predictor variable are functional. On the one hand, we generalize the results from \cite{ferraty_regression_2012} to the case of dependent data. The latter manuscript characterizes the limiting distribution of the kernel estimator in a double functional setting for pairs of independent data. On the other hand, we study the naive and the wild bootstrap in a double functional setting. Here we generalize the results of \cite{ferraty_validity_2010} as we consider a functional response variable. We provide limit theorems and characterize the consistency of the bootstrapped regression operator.

	The remainder of this manuscript is organized as follows. Section~\ref{NotationsAndHypotheses} contains the notations and hypotheses. The main results are contained in Section~\ref{MainResults} and concern the explicit form of the bias, limiting laws of the estimator and the consistency of the bootstrap approximations. The proofs are presented in Section~\ref{TechnicalResults} and mainly rely on exponential inequalities and limit theorems for Hilbert space-valued martingale differences arrays.

	\section{Notations and hypotheses}\label{NotationsAndHypotheses}
	We work on the two spaces $\cE$ and $\cH$. It is worth noting that even though in practice $\cE$ can coincide with $\cH$, we need the two different topological structures $(\cE,d)$ and $(\cH,\scalar{\cdot}{\cdot})$ in order to use the full potential of the functional kernel regression model. While the Hilbert space $\cH$ is normed, the pseudometric $d$ on $\cE$ is not necessarily a metric anyway. The choice of the pseudometric $d$ crucially influences the limiting behavior of the small ball probabilities and consequently also the rates of convergence. We shall see this in more detail below, moreover we refer to the remarks in \cite{ferraty_regression_2012}. We also refer to \cite{laib2010nonparametric} and \cite{laib2011rates} for examples of small ball probability functions of finite- and infinite-dimensional processes.
	
The functional data is ordered in a triangular array because this ensures (formally) the applicability to other types of data than time series, e.g., random fields. For that reason, let $\cS_n= ((X_{n,1},Y_{n,1}),\ldots,(X_{n,n},Y_{n,n}))$	be a functional sample with values in $\cH\times\cE$. The distribution of the pairs $(X_{n,i},Y_{n,i})$ on $\cE\times\cH$ is the same for all $1\le i \le n$ and $n\in\N$. Let $\cF_{n,i}$ be the $\sigma$-algebra generated by $(X_{n,1},Y_{n,1}),\ldots, (X_{n,i},Y_{n,i})$ and $\cG_{n,i}$ be the $\sigma$-algebra generated by $(X_1,Y_1),\ldots,(X_{n,i},Y_{n,i}), X_{n,i+1}$. The closed $\delta$-neighborhood of $x$ in $(\cE,d)$ is abbreviated by $U(x,\delta) = \{y\in\cE: d(x,y)\le \delta\}$. Write $F_x(h)$ (resp.\ $F_x^{\cF_{n,i-1}} (h)$) for the distribution function of the random variable $d(x,X_{n,i})$ (resp.\ the conditional distribution function given $\cF_{n,i-1}$), where $x\in\cE$ and $i\in\N$. 

	We write $\norm{\cdot}$ for the norm on the Hilbert space $\cH$ which is induced by $\scalar{\cdot}{\cdot}$. If $A,v\in\cH$, we also write $A_v = \scalar{A}{v}$, this abbreviation will be useful if we consider projections of $\cH$-valued functions. Moreover, if $\zeta$ is a real-valued (resp.\ $\cH$-valued) random function which satisfies $\zeta(u)/u \rightarrow 0$ $a.s.$ (resp.\ $\norm{\zeta(u)}/u \rightarrow 0$ $a.s.$) as $u\rightarrow 0$, we write $\zeta(u)= \oas(u)$. In the same way, we say that $\zeta(u)$ is $\Oas(u)$ if $\zeta(u)/u$ (resp.\ $\norm{\zeta(u)}/u$) is $a.s.$ bounded as $u\rightarrow 0$. We write $\norm{\cdot}_{\p,p}$ for the $p$-norm of a real-valued random variable w.r.t.\ the probability measure $\p$. Moreover, we abbreviate the conditional expectation (resp.\ conditional distribution) of a random variable $Z$ given the sample $\cS_n$ by $\Ec{Z}$ (resp.\ $\p^*(Z\in \cdot)$).
	
	A Borel probability measure $\mu$ on $\cH$ is a Gaussian measure if and only if its Fourier transform $\hat\mu$ is given by
	$
			\hat\mu(x) \equiv \exp(	i\scalar{m}{x} - \scalar{\cC x}{x}/2 ),
	$
	where $m\in\cH$ and $\cC$ is a positive symmetric trace class operator on $\cH$. $m$ is the mean vector and $\cC$ is the covariance operator of $\mu$. In particular, $\int_\cH \norm{x}^2 \mu(\intd{x}) = \operatorname{Tr} \cC + \norm{m}^2$. We also write $\fG(m,\cC)$ for this measure $\mu$.

	The kernel estimator is defined for a kernel function $K$, a bandwidth $h>0$ and a sample $\cS_n$ as	
	\begin{align}
	\begin{split}\label{Eq:DefEstimator}
			&\hat{r}_{n,h}(x) \coloneqq \frac{ \hat{r}_{n,h,2}(x) }{ \hat{r}_{n,h,1}(x)  }	\quad \text{ where } \hat{r}_{n,h,j}(x) \coloneqq ( n \, \E{ \Delta_{1,h,1}(x) } )^{-1}  \sum_{i=1}^n Y_{n,i}^{j-1} \Delta_{n,h,i}(x) \\
				&\qquad\qquad\qquad\qquad\qquad\qquad \text{ for } \Delta_{n,h,i}(x) = K( d(x,X_{n,i})/h ) \text{ and } j=1,2.
	\end{split}\end{align}

	Our framework corresponds largely to that in \cite{laib2010nonparametric} and \cite{laib2011rates}. However, we need at some points stricter assumptions because we consider the case where both response and predictor are of a functional nature and also study residual bootstrap procedures. We investigate the model at an arbitrary but fixed point $x\in\cE$ and assume the following hypotheses. For the sake of brevity, we give the range of the indices already at this point and omit this within the hypotheses: $y\in\cE$, $m,n,k\in \N$ while $1\le i,j\le n$ unless mentioned otherwise.
	\begin{itemize}\setlength\itemsep{0em}

	\item [\mylabel{C:Kernel}{(A1)}]
	 $K$ is a nonnegative bounded kernel of class $C^1$ over its support $[ 0 , 1 ]$. The derivative $K'$ exists on $[ 0 , 1 ]$ and satisfies the condition $K'(t)\le 0$ for all $t\in [0,1]$ and $K(1)>0$.

	\item [\mylabel{C:SmallBall}{(A2)}]		
		\begin{itemize} \setlength\itemsep{0em}
		\item [\mylabel{C:SmallBall1}{(i)}] 
			$
			 |F^{\cF_{n,i-1}}_{y_1}(u)-F^{\cF_{n,i-1}}_{y_2}(u)|/F^{\cF_{n,i-1}}_x(u) \le \tilde{L}_{n,i}\, d(y_1,y_2)^\alpha
			$
			$a.s.$ for all $y_1,y_2\in\cE$, for an $\tilde{L}_{n,i},\in\R_+$, $\alpha\in (0,1]$ uniformly in $u$ in a neighborhood of 0. Moreover, 
			$
						 [F^{\cF_{n,i-1}}_{x}(u+\veps)-F^{\cF_{n,i-1}}_{x}(u-\veps)] / F^{\cF_{n,i-1}}_x (u) \le \tilde{L}_{n,i}\, \veps\, u^{-1}
			$
			$a.s.$ at $x$ for $0\le \veps<u$ and $u$ in a neighborhood of 0. $\limsup_{n\rightarrow\infty}n^{-1}\sum_{i=1}^n \tilde{L}_{n,i}^2 = \tilde{L}<\infty$.
			\end{itemize}
	There is a sequence of nonnegative bounded random functionals $( f_{n,i , 1} )_{ 1\le i\le n}$, a sequence of random functions $( g_{n,i , y} )_{ 1\le i\le n}$, $y\in\cE$, a deterministic nonnegative bounded functional $f_1$ and a nonnegative real-valued function $\phi$ tending to 0, as its argument tends to 0 such that
	\vspace{-.5em}
	\begin{itemize} \setlength\itemsep{0em}
			\item [\mylabel{C:SmallBall2}{(ii)}] $F_y(u) = \phi(u) f_1(y) + o(\phi(u))$ ($u \rightarrow 0$), $y\in\cE$.
			
			\item [\mylabel{C:SmallBall3}{(iii)}] $F_y^{\cF_{n,i-1}}(u) = \phi(u) f_{n,i,1}(y) + g_{n,i,y}(u)$ such that $|g_{n,i,y}(u)| = \oas(\phi(u))$ and
			\begin{itemize}\setlength\itemsep{0em}
				
			\item $n^{-1} \sum_{i=1}^n |g_{n,i,y}(u)| = \oas(\phi(u))$ as $n\rightarrow \infty$.
			\item $\norm{f_{n,i,1}(y) }_{\p, \infty} \le \tilde{L}_{n,i}$ and $\phi(u)^{-1} \norm{ g_{n,i,y}(u)  }_{\p, \infty} \le \tilde{L}_{n,i}	$	 for $u \le \delta$, $\forall y\in U(x,\delta)$, for some $\delta>0$. 
				\end{itemize}		
		
			\item [\mylabel{C:SmallBall4}{(iv)}] $\lim_{n\rightarrow\infty} n^{-1} \sum_{i=1}^n f_{n,i,1}(x) = f_1(x) >0$ $a.s.$

			\item [\mylabel{C:SmallBall5}{(v)}] $\sup_{u\in [0,1] }	| \phi(hu)/\phi(h) - \tau_0(u) | = o(1)$ as $ h\downarrow 0$ for some $\tau_0\colon [0,1]\to [0,1]$.
			
	\end{itemize}

	\item [\mylabel{C:Response}{(A3)}]

	\begin{itemize}\setlength\itemsep{0em}

		\item [\mylabel{C:Response1}{(i)}] $\E{ Y_{n,i} | \cG_{n,i-1} } \equiv r(X_{n,i})$ $a.s.$
		
		\item [\mylabel{C:Response2}{(ii)}] $\E{ \scalar{e_j}{Y_{n,i}-r(X_{n,i})} \scalar{e_k}{Y_{n,i}-r(X_{n,i})} | \cG_{n,i-1} } \equiv W_{2,j,k} (X_{n,i})$ for $1\le j,k\le n$ such that
			\begin{itemize}\setlength\itemsep{0em}
			\item		$\sup_{j,k \in \N} \sup_{y\in U(x, h)  } |W_{2,j,k}(y) - W_{2,j,k}(x)| = o(1)$ 
			\item $\sup_{y \in U(x,h) }  | \sum_{k\in\N} W_{2,k,k}(y) - W_{2,k,k} (x) | = o(1) $
			\item $\sup_{y\in U(x,\delta) } \sum_{k>m } W_{2,k,k}(y) + r_{e_k}(y)^2 \le a_0 \exp(-a_1 m) $	for certain $a_0,a_1>0$.
				\end{itemize}
				
				\item [\mylabel{C:Response3}{(iii)}] $\sup_{y\in U(x,\delta) }\E{\norm{Y_{n,i}}^{(2+\delta') \cdot m} | X_{n,i}= y, \cF_{n,i-1} } \le m! \tilde{H}^{m-2}$ for all $1\le j \le 2+\delta'$ for some $\tilde{H} \ge 1$ and for some $\delta'>0$.
				
	\end{itemize}

	\item [\mylabel{C:Regressor}{(A4)}]
	\begin{itemize}\setlength\itemsep{0em}
		
		\item [\mylabel{C:Regressor1}{(i)}] $r_{e_k}$ is continuous and
		$
			\E{	r_{e_k}(X_{n,i}) - r_{e_k}(y) | \cF_{n,i-1}, d(y,X_{n,i})=s } \equiv \psi_{k,y}(s) + \bar{g}_{k,n,i,y}(s),
			$
	where $\psi_{k,y}$ is a deterministic functional and $\bar{g}_{k,n,i,y}$ is a random function. $\psi_{k,y}(0)=0$ and $\psi_{k,y}$ is differentiable in a neighborhood of $(x,0)$ with $\psi'_{k,y}(0) \neq 0$ such that 
	\begin{itemize}\setlength\itemsep{0em} 
	\item $\sup_{ u: 0\le u\le s, y\in U(x,s) } | \psi'_{k,y}(u) - \psi'_{k,x}(0)  | \, s + \norm{\bar{g}_{k,n,i,y}(s) }_{\p,\infty} \le L_k s^{1+\alpha}$ ($\alpha$ from \ref{C:SmallBall}).
	\item $\sum_{k\in\N} \psi'_{k,x}(0) ^2 + L_k^2 < \infty$.
		\end{itemize}
		
		\item [\mylabel{C:Regressor2}{(ii)}] 
		$
				\E{ |r_{e_k}(X_{n,i}) - r_{e_k}(x)|^2 | \cF_{n,i-1}, d(x,X_{n,i})=s} \equiv \psi_{2,k,x}(s) + \bar{g}_{2,k,n,i,x}(s),
		$
	where $\psi_{2,k,x}$ is a deterministic functional and $\bar{g}_{2,k,n,i,x}$ is a random function. $\psi_{2,k,x}(0)=0$ and $\psi_{2,k,x}$ is differentiable in a neighborhood of 0 with $\psi'_{2,k,x}(0) \neq 0$ such that 
	\begin{itemize}\setlength\itemsep{0em}
		
		\item	$\sup_{ \substack{u: 0\le u\le s} } |\psi'_{2,k,x}(u) - \psi'_{2,k,x}(0) | \, s + \norm{\bar{g}_{2,k,n,i,x}(s) }_{\p,\infty} \le L_{2,k} s^{1+\alpha}$ ($\alpha$ from \ref{C:SmallBall}).
		\item $\sum_{k>m} |\psi'_{2,k,x}(0)| + L_{2,k}  \le a_0 \exp(-a_1 m)$ where $a_0,a_1$ are from \ref{C:Response}.
	\end{itemize}
	\end{itemize}

	\item [\mylabel{C:Bandwidth}{(A5)}] $h,b\rightarrow 0$, $h/b\rightarrow 0$, $(n\phi(h))^{1/2} (\log n)^{-[(2+\alpha)\vee (1/2)]} \rightarrow \infty$ ($\alpha$ from \ref{C:SmallBall}), $[\phi(h)/\phi(b) ] (\log n)^2 = o(1)$, $h(n \phi(h))^{1/2} = \cO(1)$, $b^{1+\alpha} (n\phi(h))^{1/2}=o(1)$, $b = o( h \sqrt{n \phi(h)})$, $h(\log n)^{1/2}=o(1)$.

	\item [\mylabel{C:Covering}{(A6)}] For each $n$ there are $\kappa_n\in\N_+$, $\ell_n > 0$ and points $z_{n,1},\ldots,z_{n,\kappa_n}$ such that $U(x,h) \subseteq \bigcup_{u=1}^{\kappa_n} U(z_{n,u },\ell_n)$ with $\kappa_n = \cO(n^{b/h})$ and $\ell_n=o(b (n\phi(h))^{-1/2}  (\log n)^{-1} )$.

	\end{itemize}

	Moreover, we define moments which are independent of the location $x\in\cE$
	\begin{align*}
			M_0 &\coloneqq  K(1) - \int_0^1 ( K(s) s )' \tau_0(s) \intd{s} \text{ and } M_j \coloneqq K^j(1) - \int_0^1 ( K^j(s) )' \tau_0(s) \intd{s} \text{ for } j\ge 1.
	\end{align*}

\ref{C:Kernel} is very usual in nonparametric functional estimation. Since $K(1)>0$ and $K'\le 0$, $M_1>0$ for all limit functions $\tau_0$. In particular, the positivity of $K(1)$ is necessary as the small ball probability function $\tau_0$ equals the Dirac $\delta$-function at 1 in the case of non-smooth processes, see \cite{ferraty_nonparametric_2007}. So that the moments $M_j$ are determined by the value $K(1)$ in this special case.

	Assumption~\ref{C:SmallBall} is crucial for all results in this paper because it determines the limiting behavior of the kernel estimates. The functionals $f_{n,i,1}$ and $f_1$ play the role of conditional and unconditional densities. The function $\phi(u)$ characterizes the impact of the radius $u$ on the small ball probability when $u$ tends to 0. Many smooth and non-smooth processes satisfy the \ref{C:SmallBall}, see \cite{laib2010nonparametric} and \cite{laib2011rates} for some examples. Moreover, a profound survey on small ball probabilities of Gaussian processes give \cite{li2001gaussian}.

	As we also study the bootstrap in this model, it is also necessary that some assumptions in \ref{C:Response} and \ref{C:Regressor} hold uniformly in a neighborhood of the function $x$. Condition~\ref{C:Response} is a kind of Markov-type condition and characterizes the conditional means and the conditional covariance operator of the error terms. Consider a heteroscedastic regression model $Y_{n,i} = r(X_{n,i})+\varsigma(X_{n,i}) \epsilon_{n,i}$ where $\varsigma$ is real-valued and where the innovations $\epsilon_{n,i}$ are martingale differences w.r.t.\ $\cG_{n,i-1}$ with $\E{\epsilon_{e_j,n,i} \epsilon_{e_k,n,i} |\cG_{n,i-1}} \in\R$. Then, both the conditional expectation and the conditional covariance operator only depend on $X_{n,i}$ on not on further observations from the past. The tail behavior and the continuity of the covariance operator ensure almost the same rates as for a real-valued response. The moment condition on the conditional expectation of $Y$ is fairly mild.

	Assumption~\ref{C:Regressor} concerns the continuity of the components of the regression operator. A version of \ref{C:Regressor}~\ref{C:Regressor1} is also used in \cite{ferraty_regression_2012}. \ref{C:Regressor}\ref{C:Regressor1} and \ref{C:Regressor2} mean that the conditional expectation of the (squared) difference of $r_{e_k}(X_{n,i}) -r_{e_k}(x)$ given the past and the distance $d(x,X_{n,i})$ is dominated by the distance if it is small.

	The assumptions on the bandwidth in \ref{C:Bandwidth} and on the covering of the neighborhood of $x$ in \ref{C:Covering} are very similar as those in \cite{ferraty_validity_2010} who consider the bootstrap for a real-valued response. Clearly, the oversmoothing bandwidth $b$ has to satisfy $b/h \rightarrow \infty$ to make the bootstrap work. Also note that the optimal choice for the bootstrap bandwidth $h$ which is $h=\cO( (n\phi(h))^{-1/2} )$ is allowed as in \cite{ferraty_validity_2010}.

We can verify \ref{C:Covering} in the case of smooth function classes. Let $\cD$ be a compact and convex subset of $\R^d$ with nonempty interior. Let $\gamma>0$ and $\underline{\gamma}$ be the greatest integer smaller than $\gamma$. Define the differential operator 
	\[
				D^k = \frac{\partial^{k.} }{\partial t_1^{k_1} \cdots \partial t_d^{k_d} },
	\]
	where $k=(k_1,\ldots,k_d)$ and $k. = \sum_{i=1}^d k_i$. Then set for a function $x\colon\cD\rightarrow \R$ 
	\begin{align}\label{Def:NormCM}
		\norm{x}_\gamma = \max_{k. \le \underline{\gamma} } \sup_{t\in\cD} |D^k x(t)| + \max_{k.=\underline{\gamma}} \, \sup_{\substack{s,t\in \cD,\\s\neq t} } \frac{|D^k x(s)-D^k x(t)| }{\norm{s-t}_2^{\gamma-\underline{\gamma}} },
	\end{align}
	where $\norm{\cdot}_{2}$ is the Euclidean norm on $\R^d$. Let $M>0$ and define by $C^\gamma_M = \{ x:\cD\rightarrow \R, \norm{x}_\gamma \le M\}$ the class of functions that posses uniformly bounded partial derivatives up to order $\gamma$ and whose highest partial derivatives are Lipschitz-continuous of order $\gamma-\underline{\gamma}$.

	Let $\mu$ be a finite Borel measure on $\cD$ and $p\in [1,\infty]$. Write $\norm{x}_{L^p(\mu)}$ for the $L^p$-norm of a function $x\colon\cD\rightarrow \R$ and $\norm{x}_\infty$ for its supremum norm. Denote the $\veps$-covering number of the set $C^\gamma_M$ w.r.t.\ $\norm{\cdot}_\infty$ (resp.\ $\norm{\cdot}_{L^p(\mu)}$) by $N(\veps,C^\gamma_M,\norm{\cdot}_{\infty})$ (resp.\ $N(\veps,C^\gamma_M,\norm{\cdot}_{L^p(\mu)})$). It follows from Theorem 2.7.1 in \cite{van2013weak} that $\log  N(\veps,C^\gamma_1,\norm{\cdot}_{\infty})\le C  \veps^{-d/\gamma}$ for a constant $C$ which depends on the domain $\cD$. Moreover,
	\[
				N(\veps,C^\gamma_M,\norm{\cdot}_{L^p(\mu)}) \le  N(\veps \mu(\cD)^{-1/p},C^\gamma_M,\norm{\cdot}_{\infty}) \le  N(\veps \mu(\cD)^{-1/p} M^{-1},C^\gamma_1,\norm{\cdot}_{\infty}).
	\]
	Now assume that $\cE = C^\gamma_M$ for some $M,\gamma>0$ and that $d$ is the metric derived from $\norm{\cdot}_{\infty}$ or from $\norm{\cdot}_{L^p(\mu)}$. It is straightforward to demonstrate that \ref{C:Covering} is satisfied in the case that $h b^{-(1+d/\gamma)} (\log n)^{-1+d/\gamma} = o(1)$ and $\ell_n = bh (\log n)^{-1}$. See also \cite{ferraty_validity_2010} who give this condition for functions defined on an interval, i.e., $d=1$.

	Another well-known example is a separable Hilbert space $\cE$ with orthonormal basis $\{f_j:j\in\N\}$ and inner product $\scalar{\cdot}{\cdot}$. Consider the pseudometric $d(x,y) = (\sum_{j=1}^p \scalar{x-y}{f_j}^2 )^{1/2}$ for $p\in\N_+$. In this case the covering condition is satisfied if $n^{b/h} b^p (\log n)^{-p} \rightarrow \infty$, compare again \cite{ferraty_validity_2010}.

	Examples of stationary ergodic data which apply to time series are given in \cite{laib2010nonparametric}. In this paper, we focus an the spatial applications of the present framework and consider a spatial autoregressive (SAR) process.

	\begin{example}\label{Ex:SAR}[Spatial autoregressive processes]
	Let $\cD$ be a compact and convex subset of $\R^d$. Let $\cH$ be the Hilbert-space of all functions $x\colon \cD\rightarrow\R$ which are square-integrable w.r.t. Lebesgue measure. Write $\norm{\cdot}$ for the norm on $\cH$ induced by the inner product. Let $\cE$ be a subspace of $\cH$. The (pseudo-) metric $d$ is assumed to be translation invariant, which means that $d(x,z)=d(x+y,z+y)$ for all $x,y,z\in\cE$. Assume that $X$ takes values in $\cE$ and is a stationary first-order spatial autoregressive process on the lattice $\Z^2$ endowed with the four-nearest neighbor structure, viz.
	\begin{align}\label{Eq:SAR1}
				X_{(i,j)} = \theta_{1,0}( X_{(i-1,j)} ) + \theta_{0,1}( X_{(i,j-1)} ) + \epsilon_{(i,j)}, \quad (i,j)\in\Z^2,
	\end{align}
	for two operators $\theta_{1,0}, \theta_{0,1}\colon \cH\rightarrow\cH$ and for i.i.d.\ innovations $\epsilon_{(i,j)}$. The latter also take their values $\cE$.
	
	Denote the norm of a linear operator on $\cH$ by $\norm{\cdot}_{\cL(\cH)}$. If $\theta_{\ell,1-\ell}$ are linear with $\norm{\theta_{\ell,1-\ell}}_{\cL(\cH)} < 1/2$ ($\ell\in\{0,1\}$), one can show similar as in \cite{bosq_linear_2000} that $X$ has the stationary solution
	\begin{align}\label{Eq:SAR2}
				X_{(i,j)} = \sum_{u=0}^\infty \sum_{ k \in \{0,1\}^u } \left(\prod_{\ell=1}^u \theta_{k_\ell,1-k_\ell} \right) (\epsilon_{(i-k.,j-(u-k.))}),
	\end{align}
	where $k. = \sum_{\ell=1}^u k_\ell$. This series converges $a.s.$ and in $L^2_\cH(\p)$, see the continuation of this example in \ref{Pf:ExSAR}~\ref{ExA}. This also corresponds to the findings in \cite{basu1993properties} and \cite{bustos2009spatial} who study real-valued linear spatial processes. 
	
If we assume a more general situation of the model \eqref{Eq:SAR1} in which $\theta_{1,0}$ $\theta_{0,1}$ are smooth (see \eqref{C:Smoothness}) and (nonlinear) Lipschitz-continuous operators w.r.t. $\|\cdot\|$ with Lipschitz constant smaller than $1/2$, one can show that the $X_{(i,j)}$ are $\cE$-valued provided the error terms $\epsilon_{(i,j)}$ take values in a sufficiently small subspace $\cE'$. For instance, let $\cE' \subseteq C^1_{M'}$ where the latter is the space of all Lipschitz-continuous functions $x\in \cH$ with $\norm{x}_1\le M'$ for some $M'>0$ and where $\norm{\cdot}_1$ is defined in \eqref{Def:NormCM}. Then the $L^2$-norm of the functions in $\cE'$ is at most $M' |\cD|^{1/2}$, where $|\cD|$ is the Lebesgue measure of $\cD$. 
	
	Moreover, if $\ell\in\{0,1\}$ and if both operators satisfy the smoothness condition
	\begin{align}\label{C:Smoothness}
				|\theta_{\ell,1-\ell} (x)(t) - \theta_{\ell,1-\ell} (x)(s)| \le A\, (1+\norm{x})\, |t-s|,
	\end{align}
	for a certain $A\in\R_+$, for all $x\in\cH$ and all $s,t\in\cD$, then one can show that also the norms $\| X_{(i,j)} \|$, $\| X_{(i,j)} \|_{1}$ are uniformly bounded above, in particular, the $X_{(i,j)}$ take their values in $C^1_M$ for some $M\in\R_+$, see also \ref{Pf:ExSAR}~\ref{ExB} for more details.
	
	In the following, we study the conditional small ball probability structure of the SAR process. For simplicity, we do this on a rectangular index set. Let $n_1,n_2$ be two integers and consider $\cI_{n_1,n_2} = \{	(i,j)\in\Z^2: 1\le i\le n_1, 1\le j \le n_2	\}$. Write $d_k$ for the (partial) diagonal in the plane $\Z^2$ which contains all points $(i,j) \in \cI_{n_1,n_2}$ such that $i+j-1 = k$. Let $\iota\colon \{1,\ldots,n_1 n_2\}\rightarrow \cI_{n_1,n_2}$ be the enumeration of $\cI_{n_1,n_2}$ which processes these diagonals $d_k$ in decreasing order starting with the largest value of $k$ (which is $ n_1 + n_2 - 1$ in this case) and which processes the pairs within each diagonal $d_k$ in increasing $x$-coordinate dimension.
	We can use $\iota$ to construct a line in the triangular array that contains the sample data. In particular, we can construct a filtration $(\cF_k)_{1\le k \le n_1 n_2}$, where $\cF_k = \sigma\{ X_{(i,j)}: 1\le \iota^{-1} ((i,j)) \le k \}$. 
	
	Consider the small ball probability for the event $d(x,X_{(i,j)}) \le h$ for a certain point $x\in\cE$, $h>0$ and a pair $(i,j)\in\cI_{n_1,n_2}$ such that $i,j>1$. (The number of pairs $(1,j)$ and $(i,1)$ is negligible if $n_1,n_2\rightarrow\infty$.) Using the conditional structure of the SAR process and the definition of the filtration, we find for a point $\iota(k) = (i,j)$ that
	\begin{align*}
				\p( d(x,X_{(i,j)})\le h | \cF_{k-1} ) &= \p( d(x,X_{\iota(k) }) \le h  | \cF_{k-1} ) \\
				&= \p( d(x-z,\epsilon_{(i,j) })\le h ) |_{ z= \theta_{1,0}( X_{(i-1,j)} ) + \theta_{0,1}( X_{(i,j-1)} ) }.
	\end{align*}
So the asymptotic form of the small ball probability of the innovations is of interest. Let the distribution of the $\epsilon_{(i,j)}$ on $\cE$ admit a Radon-Nikod{\'y}m derivative $\kappa$ w.r.t.\ a Borel measure $\mu$ on $\cE$. The latter satisfies
\begin{itemize} \setlength\itemsep{0em}
	\item [(i)] there is a $z\in\cE$ such that $\mu( U(z,u) ) \ge \mu( U(y,u) )$ for all $y\in\cE$ and $u$ in a neighborhood of 0,
	\item [(ii)] $\lim_{u\downarrow 0} \mu(U(y,u))=0$ and $\lim_{u\downarrow 0} \mu(U(y,u) ) / \mu( U(z,u) ) = c_y \in [0,1]$ for all $y\in\cE$,
	\end{itemize}
	Moreover, if $\sup_{y\in U(x,h)} |\kappa (y) -\kappa (x)| = o(1)$ ($h\downarrow 0$), we obtain
	\begin{align*}
	\p( d(\epsilon_{(i,j)},x) \le h ) &= \kappa(x) c_x \mu( U(z,h) ) + \int_{U(x,h) } \kappa(y) - \kappa(x) \mu(\intd{y}) \\
	&\quad +  \kappa(x) \left\{\mu( U(x,h) )  - c_x \mu( U(z,h) ) \right\}.
	\end{align*}
Set now $f_\epsilon(x) = \kappa(x) c_x$ and $\phi(u) = \mu( U(z,u)) $. The last two terms on the right-hand side are then $o(\phi(h))$. We arrive at the representation
$	\p( d(\epsilon_{(i,j)},x) \le h ) = f_\epsilon(x) \phi(h) + o(\phi(h))$. If the innovations $\epsilon_{(i,j)}
$ take values in a sufficiently small subspace (e.g., as sketched above $\cE' \subseteq C^1_M$), the assumptions in \ref{C:SmallBall} can be satisfied. See also \cite{laib2011rates} who use results of \cite{lipcer1972absolute} for examples of nonsmooth processes. Hence, $X$ is stationary ergodic in the sense of the above framework if
\begin{align}
			& (n_1 n_2)^{-1} \sum_{\substack{1\le i \le n_1\\ 1\le j \le n_2}}	f_\epsilon(x - \theta_{1,0}( X_{(i-1,j)} ) - \theta_{0,1}( X_{(i,j-1)} )  ) \rightarrow \E{	f_\epsilon(x - \theta_{1,0}( X_{(0,1)} ) - \theta_{0,1}( X_{(1,0)} )  )	} \quad a.s. \nonumber
	\end{align}
whenever $n_1, n_2 \rightarrow \infty$.	The last expectation is then the function $f_1$ from \ref{C:SmallBall} evaluated at $x$.
\end{example}

	We conclude this section with some definitions, set
	\begin{align}
	\begin{split}\label{Eq:DefCondEstimator}
			&C_n(x) \coloneqq \frac{\bar{r}_{n,h,2}(x) }{\bar{r}_{n,h,1}(x) } \text{ where } \bar{r}_{n,h,j} \coloneqq \frac{\sum_{i=1}^n \E{Y_{n,i}^{j-1} \Delta_{n,h,i}(x)  | \cF_{n,i-1} } }{ n \E{\Delta_{1,h,1}(x) } }, \quad j=1,2.
		\end{split}
	\end{align}
	This enables us to define the conditional bias of the estimator by
	\begin{align}
	\begin{split}\label{Eq:DefCondBias}
			&B_n(x) \coloneqq C_n(x) - r(x) = \frac{\bar{r}_{n,h,2}(x) }{\bar{r}_{n,h,1}(x) } - r(x).
		\end{split}
	\end{align}

	\section{Main results}\label{MainResults}
We present the main results. The first is an extension of the result of \cite{laib2010nonparametric} and considers the limit distribution of the estimated regression operator at a point $x\in\cH$. A similar statement for $\cH$-valued i.i.d. pairs can be found in \cite{ferraty_regression_2012}. 

	\begin{theorem}[Asymptotic normality]\label{Thrm:GaussianLimit}
	Assume that the hypotheses \ref{C:Kernel} to \ref{C:Bandwidth} are satisfied. Then
	\begin{align}\label{Eq:GaussianLimit}
				\sqrt{ n \phi(h)} \, (\hat{r}_{n,h}(x) - r(x) - B_n(x) ) \rightarrow \fG(0,\cC_x) \text{ in distribution},
	\end{align}
	where for each $v = \sum_{k=1}^\infty \gamma_k e_k \in\cH$ the covariance operator $\cC_x$ is characterized by the condition 
	\begin{align}\label{Eq:GaussianLimit2}
			\scalar{\cC_x v}{v} = \sigma^2_{v,x} = \frac{M_2}{M_1^2 f_1(x)} \bigg\{ \sum_{j,k} \gamma_j \gamma_k W_{2,j,k}(x) \bigg\}.
	\end{align}
	\end{theorem}

	The naive and the wild bootstrap have been studied in several variants in functional regression to approximate the asymptotic distribution. However, most results are derived in the case of a real-valued response variable $Y$. The starting point for our analysis is the result from \cite{ferraty_validity_2010}. The bootstrap procedures work as follows:

	The \textit{naive bootstrap} assumes a homoscedastic model. Then for $n\in\N_+$
	\begin{enumerate}\setlength\itemsep{0em}
				\item set $\hat\epsilon_{n,i} \coloneqq Y_{n,i} - \hat{r}_{n,b}(X_{n,i})$ for $i=1,\ldots,n$ and for an oversmoothing bandwidth $b$. Define $\overline{\hat\epsilon}_n \coloneqq n^{-1} \sum_{i=1}^n \hat\epsilon_{n,i}$.
				\item  generate i.i.d.\ $\beps_{n,1},\ldots,\beps_{n,n}$ such that each $\beps_{n,i}$ is uniformly distributed on $\{ \hat\epsilon_{n,i} - \overline{\hat\epsilon}_n: i=1,\ldots,n \}$ when conditioned on the sample $\cS_n=( (X_{n,i},Y_{n,i}):i=1,\ldots,n)$.	
				\item generate bootstrap observations according to $\bY_{n,i} \coloneqq \hat{r}_{n,b}(X_{n,i}) + \beps_{n,i}$ for $i=1,\ldots,n$.
				\item define $\bootr(x) \coloneqq ( \sum_{i=1}^n \bY_{n,i} \Delta_{n,h,i}(x) ) / ( \sum_{i=1}^n \Delta_{n,h,i}(x) )$ for $x\in\cE$.
				\end{enumerate}

	The \textit{wild bootstrap}, proposed originally by \cite{wu1986jackknife}, assumes a heteroscedastic model. For that reason the definition of the bootstrap innovations in the second step has to be altered, viz., $\beps_{n,i} \coloneqq \hat{\epsilon}_{n,i} V_{n,i}$ where $V_{n,1},\ldots,V_{n,n}$ are i.i.d. real-valued, centered random variables independent of the sample $\cS_n$ which satisfy $\E{V_{n,i}^2}=1$ and $\E{|V_{n,i}|^{2+\delta'}}<\infty$ with $\delta'$ introduced in \ref{C:Response}~\ref{C:Response3}. At this step \cite{mammen1993bootstrap} proposes a resampling such that also the third conditional moment remains unchanged. The last two steps in the resampling scheme are the same as for the naive bootstrap.

	In the homoscedastic model we need an additional hypothesis concerning the distribution of the empirical residuals. Write $\epsilon_{n,i}$ for the residual $Y_{n,i}-r(X_{n,i})$. Then assume
	\begin{itemize}\setlength\itemsep{0em}
		\item [\mylabel{C:NaiveBS}{(A7)}] 	\qquad $\overline{\hat{\epsilon}}_n \rightarrow 0$ $a.s.$,\quad $n^{-1} \sum_{i=1}^n \norm{\hat\epsilon_{n,i}}^2 \rightarrow \E{\norm{\epsilon_{1,1} }^2}$ $a.s$,\quad  $n^{-1} \sum_{i=1}^n \norm{ \hat\epsilon_{n,i} }^{2+\delta'} = \Oas(1)$\\
		$\qquad$ and $n^{-1} \sum_{i=1}^n \hat{\epsilon}_{e_j,i} \hat\epsilon_{e_k,i} \rightarrow W_{2,j,k}\;  a.s. \; \forall \, j,k\in\N$,
	\end{itemize}
	where $ W_{2,j,k} = \E{ \scalar{Y_{1,1}-r(X_{1,1})}{e_j} \scalar{Y_{1,1}-r(X_{1,1})}{e_k} } = \E{ \scalar{\epsilon_{1,1}}{e_j}\scalar{\epsilon_{1,1}}{e_k} } $.

One can prove \ref{C:NaiveBS} if the estimate of the regression operator is uniformly consistent. We do not consider this issue any further in this manuscript and refer the reader to \cite{ferraty2006nonparametric} and \cite{laib2011rates} who derive uniform rates of convergence for the kernel estimator in the case of a real-valued response variable.

	Define the conditional bias of the bootstrap by
	\begin{align}\label{Eq:CondBiasBoot}
				B^*_n(x) \coloneqq \frac{\bar{r}^*_{n,b,h,2}(x) }{ \hat{r}_{n,h,1}(x) } - \hat{r}_{n,b} (x),&\\
				\text{where } \bar{r}^*_{n,b,h,2} & \coloneqq (n \E{\Delta_{1,h,1}(x) })^{-1}\sum_{i=1}^n \Ec{Y_{n,i}^{*} \Delta_{n,h,i}(x)  | \cF^*_{n,i-1} } \nonumber \\
				&= (n \E{\Delta_{1,h,1}(x) })^{-1}\sum_{i=1}^n \hat{r}_{n,b}(X_{n,i}) \Delta_{n,h,i}(x). \nonumber
	\end{align}
	We come to the second main result of this article which is the characterization of the consistency of the bootstrap. Therefore, we use that $B^*_n(x) = B_n(x) + \oas((n\phi(h))^{-1/2})$, see Lemma~\ref{L:ConvergenceBias}. Define on $\cH$ the probability measures
	\begin{align}\begin{split}\label{DefMeasures}
			&\mu_{n,x} \coloneqq \cL\left( \sqrt{n \phi(h) } \, (\hat{r}_{n,h}(x) - r(x)) \right),
	\\
			&\qquad\qquad\qquad \mu^*_{n,x}  \coloneqq \cL^*\left( \sqrt{ n \phi(h)} \, (\bootr(x) - \hat{r}_{n,b}(x) )\right).
			\end{split}	\end{align}		
	The central result of the manuscript is as follows.
	\begin{theorem}\label{Thrm:GaussianLimitBootstrap}
	Let \ref{C:Kernel} -- \ref{C:Covering} be satisfied. Let $\Psi\colon\cH\to\R$ be bounded and Lipschitz-continuous. Assume a heteroscedastic model. Then the estimator $\bootr(x)$ based on the wild bootstrap approximates $\hat{r}_{n,h}(x) $ in the sense that
	\begin{align}\label{Eq:GaussianLimitBootstrap}
				\lim_{n\rightarrow \infty}	 \left|	\int_\cH \Psi \intd{\mu^*_{n,x} } - \int_\cH \Psi \intd{\mu_{n,x} } \right| =0 \quad a.s.
	\end{align}
	Furthermore if additionally hypothesis \ref{C:NaiveBS} is satisfied, then the same statements are true for the naive bootstrap procedure in a homoscedastic model. 
	\end{theorem}

If we impose another restriction on the bandwidth, we obtain a more familiar result.
	\begin{itemize}\setlength\itemsep{0em}
		\item [\mylabel{C:ConvBW}{(A8)}]
$\lim_{n\rightarrow\infty}h (n\phi(h))^{1/2} = \bar{c} \in \R_+$. Set
$
	\bar{B}(x) = \bar{c} M_0 \, M_1^{-1} \{\sum_{k\in\N} \psi_{k,x}'(0) e_k\} \in \cH$.
\end{itemize} 

\begin{corollary}\label{CorBS}
Assume that \ref{C:Kernel} -- \ref{C:Covering} and \ref{C:ConvBW} are satisfied. Then for the wild bootstrap $\mu_{n,x} \Rightarrow \fG(\bar{B}(x),\cC_x )$ and $\mu_{n,x}^* \Rightarrow \fG(\bar{B}(x),\cC_x )$ $a.s.$ In particular, for each $v\in\cH$
	\begin{align*}
		\lim_{n\rightarrow\infty}	\sup_{z\in\R} \left| \p^*\left( \sqrt{n \phi(h) } \scalar{\bootr(x) - \hat{r}_{n,b}(x) } { v} \le z \right) - \p\left( \sqrt{n \phi(h) } \scalar{\hat{r}_{n,h}(x) - r(x) }{v} \le z \right) \right| = 0 \quad a.s.
	\end{align*}
	If additionally \ref{C:NaiveBS} is satisfied, then the same statements are true for the naive bootstrap procedure.
\end{corollary}
The statement concerning the one-dimensional projections is similar to the result in \cite{ferraty_validity_2010} who study the bootstrap for a real-valued response $Y$ and i.i.d.\ data.

	We conclude with a remark on the normalization factor $(n\phi(h))^{1/2}$. Define $\hat{F}_{n,x}(h)$ as the empirical version of $F_x(h)$ by $n^{-1} \sum_{i=1}^n \1{ d(x,X_{n,i})\le h}$. Then the above results remain valid if we replace $(n \phi(h))^{1/2}$ by $(n \hat{F}_{n,x}(h))^{1/2}$ and omit the factor $f_1(x)$ in the definition of the covariance operator $\cC_x$ in \eqref{Eq:GaussianLimit2}. Indeed, $F_x(h)/\phi(h) = f_1(x) + o(1)$ and if $\hat{F}_{n,x}(h)/F_x(h)\rightarrow 1$ as $n\rightarrow\infty$, the claim follows from Slutzky's theorem, also compare \cite{ferraty_validity_2010} and \cite{laib2010nonparametric}.

	\section{Technical results}\label{TechnicalResults}
	In order to derive the results, we need more notation. We split the difference $\hat{r}_{n,h}(x) - C_n(x)$ in a main term and a remainder, for this reason define
	\begin{align*}
				&Q_n(x) \coloneqq	(\hat{r}_{n,h,2}(x) - \bar{r}_{n,h,2}(x) ) - r(x)( \hat{r}_{n,h,1}(x) - \bar{r}_{n,h,1}(x) ) \text{ and }\\
				&R_n(x)\coloneqq - B_n(x)( \hat{r}_{n,h,1}(x) - \bar{r}_{n,h,1}(x) ).
	\end{align*}
	Then
	\begin{align}\label{Eq:DefCondBias2}
				\hat{r}_{n,h}(x) - C_n(x) = \{ Q_n(x)+R_n(x) \} / \hat{r}_{n,h,1}(x) .
	\end{align}
	It follow several technical results which are necessary for the main result.

\subsection*{General results}
A variant of the next lemma is also given in \cite{laib2010nonparametric}.
	\begin{lemma}\label{L:LL1}
	Assume that \ref{C:Kernel} and \ref{C:SmallBall} \ref{C:SmallBall2}, \ref{C:SmallBall3} and \ref{C:SmallBall5} are satisfied. Let $j\ge 1$. Then for $x\in\cE$
	\begin{enumerate}\setlength\itemsep{0em}

		\item [(i)] $\phi(h)^{-1} \, \E{ \Delta^j_{n,h,i}(x) | \cF_{n,i-1} } = M_j f_{n,i,1}(x) + \Oas\left( g_{n,i,x}(h) \phi(h) ^{-1}  \right)$. 
		
		\item [(ii)] $\phi(h)^{-1} \, \E{ \Delta^j_{1,h,1}(x) } = M_j f_{1}(x) + o(1)$.
		
		\item [(iii)] $\phi(h)^{-1} \, \E{  d(x,X_{n,i}) h^{-1} \Delta_{n,h,i}(x) \Big| \cF_{n,i-1} } = M_0 f_{n,i,1}(x) +\Oas\left( g_{n,i,x}(h) \, \phi(h)^{-1 }  \right)$.
	\end{enumerate}
				Additionally assume \ref{C:SmallBall} \ref{C:SmallBall1}, let $\ell\in\{0,1\}$ and $\veps>0$, then  
	\begin{enumerate}\setlength\itemsep{0em}
			\item [(iv)] $\sup_{y\in U(x,\veps)} (n\phi(h))^{-1} \left| \sum_{i=1}^n \E{ (d(y,X_{n,i}) h^{-1})^\ell \Delta_{n,h,i}(y) - (d(x,X_{n,i}) h^{-1})^\ell \Delta_{n,h,i}(x) \Big| \cF_{n,i-1} } \right| = \Oas(\veps^\alpha)$.
	\end{enumerate}
	\end{lemma}
	\begin{proof}
	The statements (i) to (ii)  can be found in \cite{laib2010nonparametric}. The statement (iii) follows with the same reasoning. We conclude with the proof of (iv), where we use that
\[
	\E{ (d(x,X_{n,i}) h^{-1})^\ell \Delta_{n,h,i}(x) | \cF_{n,i-1} } = K(1) F^{\cF_{n,i-1}}_x(h) - \int_0^1 (K(s)s^\ell)' F^{\cF_{n,i-1}}_x(hs) \intd{s}.
	\]
Moreover, $|F^{\cF_{n,i-1}}_y(hs)- F^{\cF_{n,i-1}}_x(hs)| / F^{\cF_{n,i-1}}_x(h) \le \tilde{L}_{n,i} \veps^\alpha$ for all $y\in U(x,\veps)$ and $s\in [0,1]$ by \ref{C:SmallBall}~\ref{C:SmallBall1}. Furthermore, $F^{\cF_{n,i-1}}_x(h) \phi(h)^{-1} \le 2 \tilde{L}_{n,i}$ also by \ref{C:SmallBall}. Combining these results, we obtain
\begin{align*}
				&\sup_{y\in U(x,\veps)} (n\phi(h))^{-1} \left| \sum_{i=1}^n \E{ (d(y,X_{n,i}) h^{-1})^\ell \Delta_{n,h,i}(y) - (d(x,X_{n,i}) h^{-1})^\ell \Delta_{n,h,i}(x) \Big| \cF_{n,i-1} } \right| \\
				&\le 2\Big( K(1) + \int_0^1 |(K(s) s^\ell)'| \intd{s} \Big) \, \veps^\alpha\, \Big( n^{-1} \sum_{i=1}^n \tilde{L}_{n,i}^2\Big).
\end{align*}
	\end{proof}

	We give an exponential inequality for a sequence of real-valued martingale differences \eqref{Eq:ConvergenceR1}, this inequality can also be found in \cite{de1999decoupling}. Similar but more general results for independent data are given in \cite{yurinskiui1976exponential}. 
	\begin{lemma}\label{ExpIneqMDS}
	Let $(Z_i:i=1,\ldots,n)$ be a martingale difference sequence of real-valued random variables adapted to the filtration $(\cF_i:i=1,\ldots,n)$. Assume that $\E{ |Z_i|^m | \cF_{i-1} } \le \frac{m!}{2} (a_i)^2 b^{m-2}$  $a.s.$ for some $b\ge 0$ and $a_i>0$ for $i=1,\ldots,n$. Then
	\[
			\p\Big(\Big|\sum_{i=1}^n Z_i\Big| \ge t \Big) \le 2 \exp\left(	- \frac{1}{2} \frac{t^2}{ \sum_{i=1}^n (a_i)^2 + bt }	\right).
	\]
	In particular, if $b=a^2$ and $a_i=a$ for $i=1,\ldots,n$, then $\p(|\sum_{i=1}^n Z_i| > n t) \le 2 \exp( 2^{-1} n t^2 (a^2(1+t))^{-1} )$.
	\end{lemma}
	\begin{proof}
	We compute the Laplace transform of $\gamma Z_i$ conditional on $\cF_{i-1}$ for $0<\gamma<b^{-1}$ and $i=1,\ldots,n$.
	\begin{align}
			\E{ \exp(\gamma Z_i) |\cF_{i-1}} &\le 1 + \sum_{m=2}^{\infty} \frac{\gamma^m}{m!} \E{ | Z_i|^m |\cF_{i-1}} \le 1 + \frac{\gamma^2 (a_i)^2 }{2} \sum_{m=2}^{\infty} (\gamma b)^{m-2} \le \exp\left( \frac{\gamma^2 (a_i)^2 }{2} \frac{1}{1-\gamma b} \right). \nonumber
	\end{align}
	Thus, 
	\[
			\p(|\sum_{i=1}^n Z_i| > t) \le 2 \exp(-\gamma t ) \exp\left( \frac{\gamma^2 \sum_{i=1}^n (a_i)^2 }{2} \frac{1}{1-\gamma b} \right).
	\]
	Setting $\gamma_0 \coloneqq t (\sum_{i=1}^n (a_i)^2 + b t)^{-1} < b^{-1}$ yields the conclusion.
	\end{proof}

The next lemma is fundamental as it studies the behavior of $(n\phi(h))^{-1} \sum_{i=1}^n Y_{v,n,i}^\ell \Delta^j_{h,n,i}(y) $ for $\ell\in\{0,1\}$ and $j\ge 1$. Here $Y_{v,n,i} = \scalar{Y_{n,i}}{v}$ is the projection of $Y_{n,i}$ in direction $v\in\cH$.
	\begin{lemma}\label{L:ConvergenceR1}
	Assume \ref{C:Kernel}, \ref{C:SmallBall} \ref{C:SmallBall2} -- \ref{C:SmallBall5} and \ref{C:Response}. Then for each $y\in U(x,\delta)$, $v\in\cH$, $j\ge 1$ and $\ell\in\{0,1\}$
	\begin{align}\begin{split}\label{Eq:ConvergenceR0}
								&\p\left(	(n\phi(h))^{-1} \left|\sum_{i=1}^n Y_{v,n,i}^\ell \Delta^j_{h,n,i}(y) - \E{ Y_{v,n,i}^\ell \Delta^j_{h,n,i}(y) \Big | \cF_{n,i-1}} \right|	> t\right) \\
								&\le 2 \exp\left( - \frac{1}{2}	\frac{t^2 n \phi(h) }{ 16 \tilde{L} K(0)^{2j}  + 2K(0)^j \tilde{H}^\ell t}	\right).
	\end{split}\end{align}
	Consequently,	$(n\phi(h))^{-1} \left|\sum_{i=1}^n Y_{v,n,i}^\ell \Delta^j_{n,h,i}(y) - \E{ Y_{v,n,i}^\ell \Delta^j_{n,h,i}(y) \Big | \cF_{n,i-1}} \right| = \Oas\left( (n \phi(h) )^{-1/2 } (\log n)^{1/2}\right)$.
	So,
	\begin{align}\label{Eq:ConvergenceR0b}
			(n\phi(h))^{-1} \sum_{i=1}^n Y_{v,n,i}^\ell \Delta_{n,h,i}^j (x) = r_v(x)^\ell M_j f_1(x) + \oas(1).
				\end{align}
	Moreover, $\lim_{n\rightarrow \infty} \hat{r}_{n,h,1}(x) = 1$ $a.s.$ and $\lim_{n\rightarrow \infty} \bar{r}_{n,h,1}(x) = 1$ $a.s.$
	\end{lemma}
	\begin{proof}
	First, consider $\ell=0$. We prove the statement concerning $\hat{r}_{n,h,1}$ and $\bar{r}_{n,h,1}$. Split $\hat{r}_{n,h,1}(x)$ in a two terms, viz.
	\begin{align}\label{Eq:ConvergenceR1}
			\hat{r}_{n,h,1}(x) &= \frac{\sum_{i=1}^n \E{\Delta_{n,h,i}(x) | \cF_{n,i-1}} }{ n \E{\Delta_{1,h,1}(x)}} + \frac{\sum_{i=1}^n \Delta_{n,h,i}(x) - \E{\Delta_{n,h,i}(x) | \cF_{n,i-1}} }{ n \E{\Delta_{1,h,1}(x)}}.
	\end{align}
	The denominator of both terms in \eqref{Eq:ConvergenceR1} can be written as $n\phi(h) (M_1 f_1(x) + o(1) )$, where $f_1(x)>0$, using that $n^{-1}\sum_{i=1}^n f_{n,i,1}(x) = f_1(x) + \oas(1)$. So that the first term converges to 1 $a.s.$ with the same arguments as in \cite{laib2010nonparametric}. We derive the exponential inequality for the second term in \eqref{Eq:ConvergenceR1} with Lemma~\ref{ExpIneqMDS} and show that it vanishes $a.s.$ 
	
	Let $y\in\cE$. Set $Z_{n,i} = \phi(h)^{-1} \left\{\Delta^j_{n,h,i}(y) - \E{ \Delta^j_{n,h,i}(y) | \cF_{n,i-1} }\right\}$. Then
	\begin{align}\label{Eq:ConvergenceR1_2}
			\E{ |Z_{n,i}|^m		\Big|\cF_{n,i-1} } \le 2^m \phi(h)^{-m} \E{ \big|\Delta^j_{n,h,i}(y)  \big| ^m		|\cF_{n,i-1} } \le 2\tilde{L}_{n,i}\phi(h)^{-(m-1)} |2K(0)^j|^m .
	\end{align}
	Thus, set $(a_i)^2 = 4 \tilde{L}_{n,i} |2K(0)^j|^2 \phi(h)^{-1}$ and $b = (2K(0)^j \phi(h)^{-1})^{m-2}$ and apply Lemma~\ref{ExpIneqMDS}. This establishes \eqref{Eq:ConvergenceR0} if $\ell=0$ and also yields the conclusion for the second term in \eqref{Eq:ConvergenceR1}.

	The statement concerning the rate of convergence follows then immediately from Proposition A.4 in \cite{ferraty2006nonparametric} which considers relations between $a.c.$ and $a.s.$ convergence. If $\ell=0$, \eqref{Eq:ConvergenceR0b} is then a simple consequence of Lemma~\ref{L:LL1}. 

	Second, let $\ell=1$. Then the statement in \eqref{Eq:ConvergenceR0} follows similarly if we use that $\E{ \norm{Y_{n,i}}^m | \cF_{n,i-1}, X_{n,i}=y} \le m! \tilde{H}^{m-2}$ uniformly in $y$ in a neighborhood of $x$. For statement in \eqref{Eq:ConvergenceR0b}, use additionally that 
	\[
		\E{ r(X_i)-r(x) | \cF_{n,i-1}, d(x,X_i)=s  } = \Oas(h) \text{ in } \cH,
		\]
		see also the proof of Lemma~\ref{L:LL3}. This finishes the proof.	
	\end{proof}

\begin{corollary}\label{Cor:ConvergenceR2}
Let \ref{C:Kernel}, \ref{C:SmallBall} \ref{C:SmallBall2} -- \ref{C:SmallBall5} and \ref{C:Response} be satisfied. Then for each $y\in U(x,\delta/2)$
\[
		(n\phi(h))^{-1} \norm{ \sum_{i=1}^n Y_{n,i} \Delta_{n,h,i}(y) - \E{ Y_{n,i} \Delta_{n,h,i}(y) | \cF_{n,i-1}} } = \Oas\left( (n \phi(h) )^{-1/2 } \, \log n \right).
\]
\end{corollary}
\begin{proof}
We have
\begin{align}
			&(n\phi(h))^{-2} \norm{ \sum_{i=1}^n Y_{n,i} \Delta_{n,h,i}(y) - \E{ Y_{n,i} \Delta_{n,h,i}(y) | \cF_{n,i-1}} }^2 \\
			\begin{split}\label{Eq:ConvergenceR21}
			&\le  (n\phi(h))^{-2} \sum_{k\le m} \left( \sum_{i=1}^n Y_{e_k,n,i} \Delta_{n,h,i}(y) - \E{ Y_{e_k,n,i} \Delta_{n,h,i}(y) | \cF_{n,i-1}} \right)^2 \\
			&\quad + (n\phi(h))^{-2} \sum_{k> m} \left( \sum_{i=1}^n Y_{e_k,n,i} \Delta_{n,h,i}(y) - \E{ Y_{e_k,n,i} \Delta_{n,h,i}(y) | \cF_{n,i-1}} \right)^2.
			\end{split}
\end{align}
We begin with the last summand in \eqref{Eq:ConvergenceR21} and obtain with assumption \ref{C:Response} \ref{C:Response2}
\begin{align*}
		&\p\left( (n\phi(h))^{-2} \sum_{k> m} \left( \sum_{i=1}^n Y_{e_k,n,i} \Delta_{n,h,i}(y) - \E{ Y_{e_k,n,i} \Delta_{n,h,i}(y) | \cF_{n,i-1}} \right)^2  > t^2 		\right) \\
			&\le (t n\phi(h))^{-2} \sum_{k> m} \E{ \left( \sum_{i=1}^n Y_{e_k,n,i} \Delta_{n,h,i}(y) - \E{ Y_{e_k,n,i} \Delta_{n,h,i}(y) | \cF_{n,i-1}} \right)^2 } \\
			&\le \frac{2}{ (n\phi(h) t )^2 } 	\E{ \sum_{i=1}^n \Delta_{n,h,i}^2(y) }  \left( \sup_{z\in U(x,\delta) }	\sum_{k>m} W_{2,k,k}(z) + r_{e_k}(z)^2 \right).
\end{align*}
Choose $m=c_1 \log n$ for $c_1$ sufficiently large and $t=c_2 (n \phi(h))^{-1/2} \log n$ for an appropriate constant $c_2$. Use the assumption of the decay of the last factor from \ref{C:Response}. Then for some $\tilde\delta>0$ this last line is $\cO( n^{-(1+\tilde\delta)} )$ and summable over $n\in\N$. In particular, the last summand in \eqref{Eq:ConvergenceR21} is $\Oas((n \phi(h))^{-1/2} \log n )$.

Next, consider the first summand in \eqref{Eq:ConvergenceR21}.
\begin{align*}
	&\p\left( (n\phi(h))^{-2} \sum_{k\le m} \left( \sum_{i=1}^n Y_{e_k,n,i} \Delta_{n,h,i}(y) - \E{ Y_{e_k,n,i} \Delta_{n,h,i}(y) | \cF_{n,i-1}} \right)^2  > t^2 		\right) \\
	&\le m \max_{k\le m} \p\left( (n\phi(h))^{-1} \left| \sum_{i=1}^n Y_{e_k,n,i} \Delta_{n,h,i}(y) - \E{ Y_{e_k,n,i} \Delta_{n,h,i}(y) | \cF_{n,i-1}} \right|  > t m^{-1/2} 		\right).
\end{align*}
An application of Lemma~\ref{L:ConvergenceR1} yields that also this last line is $\cO( n^{-(1+\tilde\delta)})$ and summable over $n\in\N$. Consequently, the first summand in \eqref{Eq:ConvergenceR21} is also $\Oas((n \phi(h))^{-1/2} \log n )$.
\end{proof}

\subsection*{Results on the estimates $\hat{r}_{n,h}$}
	\begin{lemma}\label{L:SupConvergenceR1}
	Assume that \ref{C:Kernel} and \ref{C:SmallBall} are satisfied. Then
	\begin{align*}
				&\sup_{y\in U(x,h) } (n\phi(b))^{-1} \sum_{i=1}^n |\Delta_{n,b,i}(y) - \Delta_{n,b,i}(x)| = \cO(h b^{-1} ) \quad a.s. \text{ and in the mean.}
	\end{align*}
	Moreover, $(n\phi(b))^{-1} \sum_{i=1}^n \E{ \sup_{y\in U(x,h) }|\Delta_{n,b,i}(y) - \Delta_{n,b,i}(x)| \big| \cF_{n,i-1} } = \cO(h b^{-1} )$ $a.s.$
	\end{lemma}
	\begin{proof}
	We begin with a simple upper bound on the difference $|\Delta_{n,b,i}(y) - \Delta_{n,b,i}(x)|$. There are two cases where this term can be nonzero. First, if both $x,y\in U(X_{n,i},b)$ and second if either $( x\in U(X_{n,i},b), y\notin U(X_{n,i},b) )$ or $(y\in U(X_{n,i},b), x\notin U(X_{n,i},b)$. Hence, we obtain  
	\begin{align}
	\begin{split}\label{Eq:SupConvergenceR1_3}
				\left|	\Delta_{n,b,i}(y) - \Delta_{n,b,i}(x)	\right| &\le \max_{0\le s\le 1} |K'(s)| \, \frac{h}{b} \, \1{ d(x,X_{n,i}) \le b } \\
				&\quad + K(0)  \1{ b-h < d(x,X_{n,i})\le b + h} . 
				\end{split}
	\end{align}
First, consider the case for the sum of conditional expectations. One obtains using the assumptions from \ref{C:SmallBall} that
\[
		(n\phi(b))^{-1} \sum_{i=1}^n \E{ \sup_{y\in U(x,h) }|\Delta_{n,b,i}(y) - \Delta_{n,b,i}(x)| \big| \cF_{n,i-1} } \le n^{-1} \sum_{i=1}^n 2 \tilde{L}_{n,i} \Big(\max_{0\le s\le 1} |K'(s)|  + \tilde{L}_{n,i} K(0)\Big) \frac{h}{b} 
\]
which is $\Oas(h b^{-1} )$ $a.s.$ The claim concerning the convergence in the mean is now immediate, too. It remains to prove the statement for the unconditional sum. Consider the special kernel $\bar{K}=\1{\cdot \in [0,1]}$ and the point $x$. By Lemma~\ref{L:ConvergenceR1} 
\[
		(n\phi(b+\veps))^{-1} \sum_{i=1}^n  \1{ d(x,X_{n,i})\le b + \veps } - \E{  \1{ d(x,X_{n,i})\le b + \veps } | \cF_{i=1} } = \Oas( (n\phi(b+\veps))^{-1/2} (\log n)^{1/2})
\]
for $\veps \in \{-h,0,h\}$. By \ref{C:Bandwidth} $ (n\phi(b))^{-1/2} (\log n)^{1/2}= o(h b^{-1})$. So, the unconditional statement is also true.
	\end{proof}

The next result is a generalization of Lemma 3 in \cite{laib2010nonparametric} to the double functional case.
	\begin{lemma}\label{L:LL3}
	Assume \ref{C:Kernel}, \ref{C:SmallBall}, \ref{C:Response} \ref{C:Response1} and \ref{C:Regressor} \ref{C:Regressor1} are satisfied. Then
	$$
			B_n(x) = \frac{M_0}{M_1} \Big(	\sum_{k\in\N} \psi'_{k,x}(0) e_k \Big) h + \oas( h ) \text{ in } \cH  \text{ and }  R_n(x) = \Oas\left( h (\log n)^{1/2} (n \phi(h)  )^{-1/2} \right) \text{ in } \cH .
	$$
	\end{lemma}
	\begin{proof}
	We begin with the term $B_n(x)$ from \eqref{Eq:DefCondBias}. It follows from Lemma \ref{L:ConvergenceR1} that $\bar{r}_{n,h,1}(x) \rightarrow 1$ $a.s.$ Consider
	\begin{align}
				&\bar{r}_{n,h,2}(x) - r(x) 	\bar{r}_{n,h,1}(x) \nonumber \\
				&= (n \E{\Delta_{1,h,1}(x) } )^{-1} \sum_{i=1}^n \E{ \left(	Y_{n,i} - r(x)		\right) \Delta_{n,h,i}(x) \Big| \cF_{n,i-1} } \nonumber \\
				&= (n \E{\Delta_{1,h,1}(x) } )^{-1} \sum_{k\in\N} \sum_{i=1}^n \E{ \left(	r_{e_k}(X_{n,i}) - r_{e_k}(x)		\right) \Delta_{n,h,i}(x) \Big| \cF_{n,i-1} } e_k \nonumber \\
				&= (n \E{\Delta_{1,h,1}(x) } )^{-1} \sum_{k\in\N} \sum_{i=1}^n \E{ \E{ 	r_{e_k}(X_{n,i}) - r_{e_k}(x)		\Big| d(x,X_{n,i}), \cF_{n,i-1} } \Delta_{n,h,i}(x) \Big| \cF_{n,i-1} } e_k \nonumber \\
				&= (n \E{\Delta_{1,h,1}(x) } )^{-1} \sum_{k\in\N} \sum_{i=1}^n \mathbb{E} \Bigg[ \psi'_{k,x}(0) d(x,X_{n,i})  \Delta_{n,h,i}(x) \Big| \cF_{n,i-1} \Bigg] e_k \label{Eq:LL3_1}\\
							\begin{split}\label{Eq:LL3_2}
				&\quad + (n \E{\Delta_{1,h,1}(x) } )^{-1} \sum_{k\in\N} \sum_{i=1}^n \mathbb{E} \Bigg[ \big[	(		\psi'_{k,x}(\xi_{n,i}) - \psi'_{k,x}(0) ) d(x,X_{n,i}) \\
				&\quad\qquad\qquad\qquad \qquad\qquad\qquad\qquad + \bar{g}_{k,i,x}( d(x,X_{n,i}))	\big] \Delta_{n,h,i}(x) \Big| \cF_{n,i-1} \Bigg] e_k,
				\end{split}
	\end{align}
	for a $\xi_{n,i}$ which is between 0 and $d(x,X_{n,i})$. The summand in \eqref{Eq:LL3_1} can be rewritten as
	\begin{align*}
				\left(\frac{\E{\Delta_{1,h,1}(x) }}{\phi(h)} \right)^{-1}  (n \phi(h))^{-1} \sum_{i=1}^n \E{ \Delta_{n,h,i}(x) \frac{ d(x,X_{n,i})  }{h}\Big| \cF_{n,i-1} }  \, \left(	\sum_{k\in\N} \psi'_{k,x}(0) e_k \right) h .
	\end{align*}
	It equals $M_0/M_1 \left(	\sum_{k\in\N} \psi'_{k,x}(0) e_k \right) h + \oas(h)$, using Lemma~\ref{L:LL1} and that $n^{-1}\sum_{i=1}^n f_{n,i,1}(x) \rightarrow f_1(x)>0$ $a.s.$ as well as $n^{-1}\sum_{i=1}^n g_{n,i,x} (h)\phi(h)^{-1}  = o(1)$ $a.s.$ Consider the squared $\cH$-norm of the summand in \eqref{Eq:LL3_2}; it is at most
	\begin{align*}
				 (n \E{\Delta_{1,h,1}(x) }  )^{-2} \left\{ \sum_{i=1}^n \E{\Delta_{n,h,i}(x) | \cF_{n,i-1} } \right\}^2 \, \left\{ \sum_{k\in\N} \left(  L_k h^{1+\alpha} \right)^2\right\}  = \Oas\left(h^{2(1+\alpha)}\right).
	\end{align*}
	This completes the first statement. For the second statement use that
	$
			R_n(x) = - B_n(x) (\hat{r}_{n,h,1}(x)-\bar{r}_{n,h,1}(x))
	$
	and combine the first statement of this lemma with Lemma~\ref{L:ConvergenceR1} which states that the difference $\hat{r}_{n,h,1}(x)-\bar{r}_{n,h,1}(x)$ is $\Oas\left( (n\phi(h))^{-1/2} (\log n)^{1/2}  \right)$.
	\end{proof}

	\begin{lemma}\label{L:NormalityQ}
	Assume \ref{C:Kernel} -- \ref{C:Regressor} for $x\in \cE$. Then
	$
			\sqrt{n\phi(h)} Q_n(x) \rightarrow \fG(0,\cC_x) \text{ in law},
$
	where the covariance operator $\cC_x$ is characterized by \eqref{Eq:GaussianLimit2}.
	\end{lemma}
	\begin{proof}
	Note that $\sqrt{n \phi(h)} Q_n(x)$ can be rewritten as $\sum_{i=1}^n \xi_{n,i}$, where $\xi_{n,i} = \eta_{n,i} - \E{\eta_{n,i}|\cF_{n,i-1}}$ is an array of martingale differences in $\cH$ with respect to $(\cF_{n,i-1}: i=1,\ldots,n)$ for the random variables
	$$
				\eta_{n,i} = \sqrt{ \frac{ \phi(h) }{n} } \left( Y_{n,i} - r(x)	\right) \frac{\Delta_{n,h,i}(x)}{\E{ \Delta_{n,h,i}(x) } }.
	$$
	Note that $\E{\norm{\xi_{n,i}}^2}<\infty$. In order to establish the asymptotic normality, we use a generalization of the Lindeberg condition for Hilbert-space valued martingale difference arrays from \cite{kundu2000central}. 
	\begin{condition}[Lindeberg condition]\label{C:Lindeberg}\ \\
	\vspace{-1.5em}
	\begin{enumerate}[label=\textnormal{(\roman*)}]\setlength\itemsep{0em}

		\item \label{C:Lindeberg1} $\lim_{n\rightarrow \infty} \sum_{i=1}^n \E{ \scalar{\xi_{n,i}}{v}^2 | \cF_{n,i-1} } = \sigma^2_v(x)$ for some $\sigma_v\in\R_+$ in probability, for every $v\in\cH$. 
		
		\item \label{C:Lindeberg2} $\lim_{n\rightarrow \infty} \sum_{k\in\N } \sum_{i=1}^n \E{ \scalar{\xi_{n,i}}{e_k}^2 } = \sum_{k\in\N} \sigma^2_{e_k}(x) < \infty$ 
		
		\item	\label{C:Lindeberg3}
		$
				\sum_{i=1}^n \E{ \scalar{\xi_{n,i}}{e_k}^2 \1{| \scalar{\xi_{n,i}}{e_k}| > \rho } | \cF_{n,i-1} }\rightarrow 0
				$
				 in probability for every $\rho>0$ and every $k\ge 1$.
	\end{enumerate}
	\end{condition}
	If these criteria are satisfied, then $\sum_{i=1}^n \xi_{n,i} \rightarrow \fG(0,\cC_x)$ in distribution where the covariance operator $\cC_x$ is characterized by the condition $\scalar{\cC_x v}{v} = \sigma_v^2(x)$. We begin with (i). Let $v=\sum_{k\in\N} \gamma_k e_k \in\cH$ be arbitrary but fixed. Write $\eta_{v,n,i}$ (resp.\ $\xi_{v,n,i}$) for $\scalar{\eta_{n,i}}{v}$ (resp.\ $\scalar{\xi_{n,i}}{v}$). Similar as in \cite{laib2010nonparametric}, we make use of the inequality
	$$
			\left| \sum_{i=1}^n \E{\eta_{v,n,i}^2 | \cF_{n,i-1} } - \E{\xi_{v,n,i}^2 | \cF_{n,i-1} } \right| \le \sum_{i=1}^n 	\E{\eta_{v,n,i} | \cF_{n,i-1} }^2.
	$$
	We show that the right-hand-side converges to 0 $a.s.$ W.l.o.g.\ assume that $d(x,X_{n,i})\le 1$ and consider
	\begin{align*}
				\left| \E{ r_v(X_{n,i})-r_v(x) | d(x,X_{n,i})=s, \cF_{n,i-1} } \right| &= \left| \sum_{k=1}^n \gamma_k \, \E{ r_{e_k}(X_{n,i})-r_{e_k}(x) | d(x,X_{n,i})=s, \cF_{n,i-1} } \right| \\
				&\le \left| \sum_{k=1}^n  \gamma_k \left(	\psi'_{k,x}(0) d(x,X_{n,i}) + \cO\left( L_k d(x,X_{n,i})^{1+\alpha}		\right)	\right) \right| \\
				&\le C \left(\sum_{k=1}^n \gamma_k^2 \right)^{1/2} \left( \left(\sum_{k=1}^n \psi'_{k,x}(0)^2 \right)^{1/2} + \left(\sum_{k=1}^n L_k^2 \right)^{1/2} 	\right) d(x,X_{n,i}),
	\end{align*}
	for a certain constant $C \in \R_+$. And by assumption $\sum_{k\in\N} \gamma_k^2 + {\psi'}_{k,x}(0)^2 + L_k^2 < \infty$. Consequently,
	\begin{align*}
			\left| \E{ \eta_{v,n,i} | \cF_{n,i-1} } \right| &= \frac{\sqrt{ \phi(h)/n } }{ \E{\Delta_{1,h,1}(x) } } \left|	\E{ (r_v(X_{n,i})-r_v(x)) \Delta_{n,h,i}(x) | \cF_{n,i-1} }	\right| \\
			&=\cO\left( \phi(h)^{1/2} \, n^{-1/2} \, h \, \E{\Delta_{1,h,1}(x) }^{-1} 	\E{  \Delta_{n,h,i}(x) | \cF_{n,i-1} } \right).
	\end{align*}
	In particular, if we use that $f_{n,i,1}(x) \le \tilde{L}_{n,i}$ and $g_{n,i,x}(h) \phi(h)^{-1} \le \tilde{L}_{n,i}$ as well as $\limsup_{n\rightarrow\infty} n^{-1} \sum_{i=1}^n \tilde{L}_{n,i}^2 <\infty$, we obtain
	$
			\sum_{i=1}^n \E{ \eta_{v,n,i} | \cF_{n,i-1} }^2 = \Oas( \phi(h) h^2 ) $. Thus, it suffices to consider the $\eta_{v,n,i}$ instead and demonstrate
	$
				\lim_{n\rightarrow\infty} \sum_{i=1}^n \E{ \eta_{v,n,i}^2 | \cF_{n,i-1} } = \sigma_v^2(x) 
	$ in probability. Observe that
	\begin{align}
	\begin{split}\label{Eq:NormalityQ1}
					\sum_{i=1}^n \E{ \eta_{v,n,i}^2 | \cF_{n,i-1} } &= \frac{ \phi(h) } { n \E{\Delta_{1,h,1}(x) }^2 }  \sum_{i=1}^n \E{ \Delta_{n,h,i}(x)^2 \E{(Y_{v,i}-r_v(X_{n,i}))^2 | \cG_{n,i-1} } |\cF_{n,i-1}  } \\
						&\quad + 	\frac{ \phi(h) } { n \E{\Delta_{1,h,1}(x) }^2 }  \sum_{i=1}^n \E{ \Delta_{n,h,i}(x)^2 (r_v(X_{n,i})-r_v(x) )^2 |\cF_{n,i-1}  } =: J_{1,n} + J_{2,n}.
						\end{split}
	\end{align}
 Next, use that $n^{-1}\sum_{i=1}^n f_{n,i,1}(x) \rightarrow f_1(x)$ $a.s.$ for both terms $J_{1,n}$ and $J_{2,n}$.	Then there is a $C\in\R_+$ such that the second term is bounded above by
	\[
			J_{2,n} \le C \norm{v}^2	\Big(\sum_{k\in\N} \psi_{2,k,x}'(0) + L_{2,k}	\Big) \frac{M_2 f_1(x) + \oas(1)}{ M_1^2 f_1(x)^2 + o(1)}\, h  = \Oas(h),
	\]
	using the assumptions on the functions $\psi_{2,k,x}$ and $\bar{g}_{2,k,i,x}$ from \ref{C:Regressor}~\ref{C:Regressor2}.
	Consider the first summand in \eqref{Eq:NormalityQ1} and use the assumptions on the family of operators $W_{2,j,k}$ from \ref{C:Response}~\ref{C:Response1} and \ref{C:Response2}. Then
	\begin{align}
					J_{1,n} &= \frac{ \phi(h) } { n \E{\Delta_{1,h,1}(x) }^2 }  \sum_{i=1}^n \E{ \Delta_{n,h,i}(x)^2 \Big( \sum_{j,k} \gamma_j \gamma_k W_{2,j,k}(x) + \norm{v}^2 o(1) \Big) \Big|\cF_{n,i-1}  }  \nonumber \\
					& \rightarrow \frac{M_2}{M_1^2 f_1(x) } \Big(\sum_{j,k} \gamma_j \gamma_k W_{2,j,k}(x) \Big) = \sigma_v^2(x) \quad a.s. \label{Eq:NormalityQ2}
	\end{align}	
	So, it remains to show that the last limit in \eqref{Eq:NormalityQ2} is meaningful. Indeed, it follows from the definition that
	\begin{align*}
				\sum_{j,k} \gamma_j \gamma_k W_{2,j,k}(x) &= \E{ \sum_{j,k} \gamma_j \gamma_k (Y_{e_j,n,i}-r_{e_j}(X_{n,i}))(Y_{e_k,n,i}-r_{e_k}(X_{n,i})) \Bigg| \cG_{n,i-1}, X_{n,i} = x } \\
				&= \E{ \left( \sum_{j} \gamma_j  (Y_{e_j,n,i}-r_{e_j}(X_{n,i})) \right)^2 \Biggl| \cG_{n,i-1}, X_{n,i} = x }  \\
				&\le \norm{v}^2 \E{ \norm{Y_{n,i}-r(X_{n,i})}^2 \Big|\cG_{n,i-1}, X_{n,i} = x } < \infty.
	\end{align*}
	The property (ii) from the Lindeberg condition follows similarly as (i). Again use 
	\begin{align*}
			\left| \sum_{k\in\N} \sum_{i=1}^n \E{ \eta_{e_k,n,i}^2 - \xi_{e_k,n,i}^2		} \right| &\le  \sum_{k\in\N} \sum_{i=1}^n \E{  \E{\eta_{e_k,n,i} | \cF_{n,i-1}}^2 } \\
			&= \cO\left\{ \left( \sum_{k\in\N} \psi'_{k,x}(0)^2 + L_k^2		\right) \frac{\phi(h)}{\E{\Delta_{1,h,1}(x) }^2 } \E{ \Delta_{1,h,1}(x)^2	} h^2 \right\} = \cO(h^2).
	\end{align*}
	Thus, it suffices to consider
	\begin{align}\begin{split}\label{Eq:NormalityQ2b}
			\sum_{k\in\N} \sum_{i=1}^n \E{ \eta_{e_k,n,i}^2} &= \frac{\phi(h)}{\E{\Delta_{1,h,1}(x)}^2 } \E{\sum_{k\in\N} (Y_{e_k,n,i} - r_{e_k}(X_{n,i}))^2 \Delta_{1,h,1}(x)^2 } \\
			&\quad+ \frac{\phi(h)}{\E{\Delta_{1,h,1}(x)}^2 } \E{\sum_{k\in\N} ( r_{e_k}(X_{n,i}) - r_{e_k}(x) )^2 \Delta_{1,h,1}(x)^2 }. 
	\end{split}\end{align}
	Using the assumptions on the functions $\psi_{2,k,x}$, one can again show that the second term \eqref{Eq:NormalityQ2b} is $\cO(h)$. While one finds that the first term in \eqref{Eq:NormalityQ2b} behaves as
	\[
		 \frac{\phi(h)}{\E{\Delta_{1,h,1}(x)}^2 } \E{ \Delta_{1,h,1}(x)^2 } \Big\{	\sum_{k\in\N} W_{2,k,k}(x) + o(1)	\Big\} \rightarrow \sum_{k\in\N} \sigma_{e_k}^2 (x) < \infty.
	\]
	This proves (ii). Finally, we verify the Lindeberg condition (iii) for each projection. Therefore, we proceed similarly as in the finite-dimensional case in \cite{laib2010nonparametric}. Let $\rho>0$ and use that 
	\begin{align}\label{Eq:NormalityQ3}
		\E{ \xi_{v,n,i}^2 \1{|\xi_{v,n,i}| > \rho} |\cF_{n,i-1} } &\le 4 \E{  |\eta_{v,n,i}|^2 \1{|\eta_{v,n,i}| > \rho/2}  |\cF_{n,i-1}} \nonumber \\
		&\le  4 \E{  |\eta_{v,n,i}|^{2a} |\cF_{n,i-1} } \left( \frac{\rho}{2} \right)^{-2a/b},
		\end{align}
	where the numbers $a,b \ge 1$ satisfy $a^{-1}+b^{-1}=1$. We use that $y\mapsto \E{ |Y_{v,n,i}|^{2+\delta'} | X_{n,i} = y, \cF_{n,i-1} }$ is bounded uniformly in a neighborhood of $x$ by \ref{C:Response}~\ref{C:Response3}. Then choose $a = 1 + \delta'/2$ for $\delta'$ from \ref{C:Response}~\ref{C:Response3}. We obtain that \eqref{Eq:NormalityQ3} is at most (modulo a constant)
	\begin{align*}
			&\left(\frac{\phi(h)}{n}\right)^{(2+\delta')/2} \frac{1}{\E{\Delta_{1,h,1}(x)}^{2+\delta'} } \E{ |Y_{v,n,i}-r_v(x)|^{2+\delta'} \Delta_{n,h,i}(x)^{2+\delta'} | \cF_{n,i-1} } \\
			&\le C \left(\frac{\phi(h)}{n}\right)^{(2+\delta')/2} \frac{\norm{v}^{2+\delta'} }{\E{\Delta_{1,h,1}(x)}^{2+\delta'} } \E{  \Delta_{n,h,i}(x)^{2+\delta'} | \cF_{n,i-1} }.
	\end{align*}
	Thus, using the convergence results implied by Lemma~\ref{L:LL1}, we obtain that
	$
			\sum_{i=1}^n \E{ \xi_{v,n,i}^2 \1{|\xi_{v,n,i}| > \rho} |\cF_{n,i-1} } = \Oas\left(	(n \phi(h))^{-\delta'/2}\right)$.
	This demonstrates (iii) and completes the proof.
	\end{proof}

	\begin{proof}[Proof of Theorem~\ref{Thrm:GaussianLimit}]
	We use the decomposition from \eqref{Eq:DefCondBias} and \eqref{Eq:DefCondBias2} to write
	\begin{align}\label{Eq:GaussianLimit1}
				\sqrt{n \phi(h)} \left(	\hat{r}_{n,h}(x) - r(x)	\right)  = \sqrt{n \phi(h)} \frac{Q_n(x) + R_n(x) }{\hat{r}_{n,h,1} } + \sqrt{n \phi(h)}  B_n(x).
	\end{align}
	Lemma~\ref{L:ConvergenceR1} states that $\hat{r}_{n,h,1}$ converges $a.s.$ to 1. Combining this result with Lemma~\ref{L:NormalityQ} yields that $\sqrt{n \phi(h)} Q_n(x)/\hat{r}_{n,h,1}$ converges to the Gaussian distribution $\fG(0,\cC_x)$. Hence, it remains to show that the remaining terms are negligible resp.\ at least bounded. We deduce from Lemma~\ref{L:LL3} that $\sqrt{n \phi(h) }  R_n(x) $ is $\Oas(h \sqrt{\log n})=\oas(1)$. Furthermore, $\sqrt{n \phi(h)}  B_n(x) $ is $\Oas\left( h	\sqrt{n \phi(h)}	\right)$.
	\end{proof}

\subsection*{Results on the bootstrap}
	\begin{lemma}\label{L:UnifExpIneqInnovations}
	Assume \ref{C:Kernel} -- \ref{C:Covering} and let $v\in\cH$. Then for $\ell\in\{0,1\}$
	\begin{align*}
		&\sup_{y\in U(x, h) } (n\phi(b))^{-1} \Big| \sum_{i=1}^n Y_{v,n,i}^\ell (\Delta_{n,b,i}(y)- \Delta_{n,b,i}(x))  \\
		&\qquad\qquad\qquad\qquad- \E{Y_{v,n,i}^\ell (\Delta_{n,b,i}(y)- \Delta_{n,b,i}(x))  | \cF_{n,i-1} } \Big| =  o\big( (n \phi(h))^{-1/2} \big) \quad a.s.
	\end{align*}
	\end{lemma}
	\begin{proof}
	 We only prove the statement for $\ell=1$. Set
	\[
	Z_{v,n,i}(y,x) =  \phi(b)^{-1} Y_{v,n,i} (\Delta_{n,b,i}(y)- \Delta_{n,b,i}(x)) - \phi(b)^{-1} \E{ Y_{v,n,i} (\Delta_{n,b,i}(y)- \Delta_{n,b,i}(x)) | \cF_{n,i-1} }.
	\]
	First, we prove the following exponential inequality which holds for $v\in\cH$ and $y\in U(x,h)$.
	\begin{align}
	\begin{split}\label{E:ExpIneqR2II0}
				\p\left(	 n^{-1} \left|\sum_{i=1}^n Z_{v,n,i}(y,x) \right| \ge t	\right) 	&\le 2 \exp\left(	- \frac{1}{2} \frac{t^2 n \phi(b) }{A_1 h b^{-1} + B_1 t }\right)
	\end{split}\end{align}
	for certain $A_1,B_1\in\R_+$ (independent of the choice of $v$ and $y$). We proceed as in the proof of Lemma~\ref{L:SupConvergenceR1} to obtain
	\begin{align*}
			\E{ |Z_{v,n,i}(y,x) |^m | \cF_{n,i-1}}
			&\le (4 K(0) \phi(b)^{-1} \tilde{H}) ^{m-2} m! \cdot 2^4 \Big(\max_{0\le s\le 1} |K'(s)|  + \tilde{L}_{n,i} K(0)\Big)  \tilde{L}_{n,i} h b^{-1} \phi(b)^{-1}.
	\end{align*}
 An application of Lemma~\ref{ExpIneqMDS} yields \eqref{E:ExpIneqR2II0}. Second, consider a covering of $U(x,h)$ with $\kappa_n$ balls of diameter $\ell_n$ centered at points $z_{n,u}$ for $u=1,\ldots,\kappa_n$ as in \ref{C:Covering}. Then
	\begin{align}		\begin{split}\label{E:UnifExpIneqR2_1}
			\sup_{y\in U(x,h) } \left|n ^{-1} \sum_{i=1}^n Z_{v,n,i}(y,x) \right| 
			&	\le  \max_{1 \le u \le \kappa_n }\left| n^{-1} \sum_{i=1}^n Z_{v,n,i}(z_{n,u},x) \right| \\
			&\quad + \max_{1 \le u \le \kappa_n } \sup_{y \in U(z_{n,u}, \ell_n)  }\left| n^{-1} \sum_{i=1}^n  Z_{v,n,i}(y,z_{n,u})   \right|.
	\end{split}\end{align}
	We begin with the first term in \eqref{E:UnifExpIneqR2_1}. Using the inequality from \eqref{E:ExpIneqR2II0}, we obtain
	\begin{align}\begin{split}\label{E:UnifExpIneqR2_1b}
			\p\left(	\max_{1 \le u \le \kappa_n }\Big| n^{-1} \sum_{i=1}^n Z_{v,n,i}(z_{n,u},x)  \Big|  > t	\right) 
			&\le 2\kappa_n  \exp\left(	- \frac{1}{2} \frac{t^2 n \phi(b) }{ A_1 h b^{-1} + B_1  t }\right).
	\end{split}\end{align}
By assumption, $\kappa_n = \cO( n^{b/h} )$ and $b = \cO( h \sqrt{n \phi(h)})$. Thus, if we choose $t = c \sqrt{ \log n / (n \phi(b)) }$ for a certain $c>0$, \eqref{E:UnifExpIneqR2_1b} is summable over $n\in\N$. This implies in particular that the first maximum on the right-hand-side of \eqref{E:UnifExpIneqR2_1} is $\Oas(\sqrt{ \log n / (n \phi(b)) } ) = \oas( (n \phi(h))^{-1/2} )$.

	Next, the second term in \eqref{E:UnifExpIneqR2_1} can be bounded above with similar arguments as those used in the derivation of \eqref{Eq:SupConvergenceR1_3} in Lemma~\ref{L:SupConvergenceR1}. Set
	\[
	\tilde{Z}_{v,n,i}(u) \coloneqq \phi(b)^{-1} |Y_{v,n,i}| \left(	\ell_n b^{-1} \1{d(z_{n,u},X_{n,i}) \le b} + \1{ b- \ell_n < d(z_{n,u},X_{n,i}) \le b + \ell_n} \right)
	\]
	for $1\le u \le \kappa_n$. Then
	\begin{align}
			&\max_{1 \le u \le \kappa_n} \sup_{y \in U(z_{n,u}, \ell_n) }\left| (n\phi(b))^{-1} \sum_{i=1}^n Y_{v,n,i} (\Delta_{n,b,i}(y)- \Delta_{n,b,i}(z_{n,u})) \right|  \le \max_{1 \le u \le \kappa_n} C n^{-1} \sum_{i=1}^n \tilde{Z}_{v,n,i}(u)  \label{E:UnifExpIneqR2_2}
	\end{align}
	for a certain constant $C\in\R_+$. $\E{ |Y_{v,n,i}| | X_{n,i}= y, \cF_{n,i-1} } $ is bounded in a neighborhood of $x$ by \ref{C:Response}. We obtain
	\begin{align*}
			&\max_{1 \le u \le \kappa_n} n^{-1} \sum_{i=1}^n \E{	\tilde{Z}_{v,n,i}(u) \big|\cF_{n,i-1} } = \Oas( \ell_n / b).
	\end{align*}
	Moreover, arguing as in the derivation of the first exponential inequality in this proof, there are $A_2,B_2\in\R_+$ such that
	\[
			\p\left( \max_{1 \le u \le \kappa_n} n^{-1} \Big| \sum_{i=1}^n 	\tilde{Z}_{v,n,i}(u) - \E{	\tilde{Z}_{v,n,i}(u) \big|\cF_{n,i-1} } \Big| \ge t \right) \le 2 \kappa_n \exp\left(-\frac{1}{2} \frac{t^2 n \phi(b)}{A_2 \ell_n b^{-1}	+ B_2 t }\right),
	\]
	this last upper bound is dominated by that in \eqref{E:UnifExpIneqR2_1b}. Consequently \eqref{E:UnifExpIneqR2_2} is $\oas( (n\phi(h))^{-1/2} )$ and the same is true for second term in \eqref{E:UnifExpIneqR2_1}.
	\end{proof}

	\begin{lemma}\label{L:UniformExpIneqInnovationsHilbert}
	Assume \ref{C:Kernel} -- \ref{C:Covering}. Then 
	\begin{align*}
			&\sup_{y\in U(x,h) } (n\phi(b))^{-1} \Big\{ Y_{n,i} (\Delta_{n,b,i}(y)-\Delta_{n,b,i}(x) )  - \E{Y_{n,i} (\Delta_{n,b,i}(y)-\Delta_{n,b,i}(x)) |\cF_{n,i-1} } \Big\} = \oas((n\phi(h))^{-1/2}).
			\end{align*}
	\end{lemma}
	\begin{proof}
	We begin with the fundamental decomposition of the Hilbert space-valued sequence $\sum_{i=1}^n Z_{n,i}$ where
	\[
		Z_{n,i}(y,x)= (n\phi(b))^{-1} \left\{ Y_{n,i} (\Delta_{n,b,i}(y)-\Delta_{n,b,i}(x) ) - \E{Y_{n,i} (\Delta_{n,b,i}(y)-\Delta_{n,b,i}(x)) |\cF_{n,i-1} } \right\}.
		\]
	Write again $Z_{v,n,i}$ for $\scalar{Z_{n,i}}{v}$, for $v\in\cH$. For every $m\in\N$ 
	\begin{align}\label{E:UniformExpIneqInnovationsHilbert1}
				\norm{ \sum_{i=1}^n Z_{n,i}(y,x) }^2  &= \sum_{k\le m } \Big|\sum_{i=1}^n Z_{e_k,n,i}(y,x)\Big|^2  + \sum_{k>m}\Big|\sum_{i=1}^n Z_{e_k,n,i}(y,x)\Big|^2 .
	\end{align}
	We set $m = (2+\tilde{\delta}) a_1^{-1} \log n$ and $t = c (n\phi(b))^{-1/2} \log n$ for some $\tilde{\delta}>0$, with $a_1$ from \ref{C:Response}~\ref{C:Response2} and for some constant $c$ which will be characterized below. Consider the second double sum in \eqref{E:UniformExpIneqInnovationsHilbert1}. 
	We will use that $\E{ |Y_{e_k,n,i}|^2 | \cF_{n,i-1}, X_{n,i} }$ is bounded above by $W_{2,k,k}(X_{n,i}) + r(X_{n,i})^2$. We have
	\begin{align*}
		&\p\left( \sup_{y\in U(x,h) }  \sum_{k>m}\Big|	\sum_{i=1}^n Z_{e_k,n,i}(y,x)\Big|^2  > t^2 \right) \\
		&\le \frac{1}{t^2}	\E{ \sup_{y\in U(x,h) }  \sum_{k>m}\Big|	\sum_{i=1}^n Z_{e_k,n,i}(y,x)\Big|^2 } \\
		&\le \frac{2}{\phi(b)^2 t^2} \sum_{k>m} \E{  Y_{e_k,n,i}^2 \sup_{y\in U(x,h) } \Big|	\Delta_{n,b,i}(y) - 	\Delta_{n,b,i}(x)  \Big|^2	} \\
		&\le \frac{C n}{c^2 \phi(b) (\log n)^2 } \sum_{k>m}\E{  \E{ Y_{e_k,n,i}^2 | \cF_{n,i-1}, X_{n,i}}  \sup_{y\in U(x,h) }  |\Delta_{n,b,i}(y) - 	\Delta_{n,b,i}(x) | } \\
		&\le \cO\left( \frac{  n}{ (\log n)^2 }  \frac{h}{b} \sup_{y\in U(x,2h) } \sum_{k>m} W_{2,k,k}(y) + r_{e_k}(y)^2 \right),
	\end{align*}
	Using the definition of $m$ and the uniform decay of the functions $W_{2,k,k}$ and $r_{e_k}$, we find that this last bound is $\cO( n^{-(1+\tilde{\delta})})$. Consequently, for the current choices of $m$ and $t$
	\[
			\sum_{n\in\N} \p\left(	 \sup_{y\in U(x,h) } \sum_{k>m}\Big|	\sum_{i=1}^n Z_{e_k,n,i}(y,x)\Big|^2 > t^2\right) = \cO \left( \sum_{n\in\N} 
		 n^{-(1+\tilde{\delta})} \right) < \infty.
	\] 
	In particular, $ \sup_{y\in U(x,h) }  \sum_{k>m}\sum_{i=1}^n Z_{e_k,n,i}(y,x)e_k = \Oas( (n\phi(b))^{-1/2} \log n ) = \oas((n \phi(h))^{-1/2})$ in $\cH$. 
	
	It remains to consider the finite-dimensional term in \eqref{E:UniformExpIneqInnovationsHilbert1}. We use the same covering as in Lemma~\ref{L:UnifExpIneqInnovations}. This term is bounded above by
	\begin{align}\begin{split}\label{E:UniformExpIneqInnovationsHilbert2}
			 \sup_{y\in U(x,h) }  \sum_{k\le m } \big|\sum_{i=1}^n Z_{e_k,n,i}(y,x)\big|^2 &\le 2 \sum_{k\le m } \max_{1\le u \le \kappa_n } \big|\sum_{i=1}^n Z_{e_k,n,i}(z_{n,u},x)\big|^2 \\
			&\quad +  2 \sum_{k\le m } \max_{1\le u \le \kappa_n } \sup_{y\in U(z_{n,u},\ell_n) } \big|\sum_{i=1}^n Z_{e_k,n,i}(y,z_{n,u})\big|^2 .
			\end{split}
	\end{align}
	Consider the first summand in \eqref{E:UniformExpIneqInnovationsHilbert2}. We obtain as in Lemma~\ref{L:UnifExpIneqInnovations}
	\begin{align}
			&\p\left(	\sum_{k\le m } \max_{1\le u \le \kappa_n } \big|\sum_{i=1}^n Z_{e_k,n,i}(z_{n,u},x)\big|^2  > t^2	\right)
			 \le 2 m \kappa_n \exp\left\{	- \frac{1}{2} \frac{t^2 m^{-1} n \phi(b) }{  A h b^{-1} + B t m^{-1/2} }	\right\} \nonumber \\
			&= \cO\left(	\log n \cdot n^{b/h} \exp\left( \frac{-c \log n}{ A h b^{-1} + B ( n \phi(b))^{-1/2} \log n^{1/2 }  }	\right)\right) \label{E:UniformExpIneqInnovationsHilbert3},
	\end{align}
	where by assumption $\kappa_n$ is $\cO( n^{b/h})$. Consequently, \eqref{E:UniformExpIneqInnovationsHilbert3} is summable over $n\in\N$ if $c$ is sufficiently large and if we use that $b (h \sqrt{n\phi(h)})^{-1}$ is bounded above. In particular, the first summand in \eqref{E:UniformExpIneqInnovationsHilbert2} is $\Oas( (n\phi(b))^{-1/2} \log n) = \oas( (n\phi(h))^{-1/2} )$.
	It remains the second summand in \eqref{E:UniformExpIneqInnovationsHilbert2} which can also be treated as in Lemma~\ref{L:UnifExpIneqInnovations}, compare \eqref{E:UnifExpIneqR2_1} and \eqref{E:UnifExpIneqR2_2} to see that the upper bounds are uniform in $e_k$. This summand is $\Oas( (\ell_n b^{-1} )^2 \log n ) + \Oas(   (n\phi(b))^{-1} (\log n)^2 ) = \oas( (n\phi(h))^{-1} )$. So, $ \sup_{y\in U(x,h) }  \sum_{k\le m}\sum_{i=1}^n Z_{e_k,n,i}(y,x)e_k = \oas(  (n\phi(h))^{-1/2} )$.
	\end{proof}

	\begin{lemma}\label{L:ConvergenceBiasIngredients} Assume \ref{C:Kernel} -- \ref{C:Covering}. Then 
	\begin{align*}
			&(n\phi(h))^{-1} \sum_{i=1}^n (r(X_{n,i})-r(x))\Delta_{n,h,i}(x) - \E{ (r(X_{n,i})-r(x))\Delta_{n,h,i}(x) | \cF_{n,i-1} } = \oas( (n\phi(h))^{-1/2} ).
	\end{align*}
	\end{lemma}
	\begin{proof}
	The proof is very similar to the proof of Lemma~\ref{L:UniformExpIneqInnovationsHilbert}. First, set 
	\[
	 Z_{n,i} \coloneqq (n\phi(h))^{-1}  \left\{ (r(X_{n,i})-r(x))\Delta_{n,h,i}(x) - \E{ (r(X_{n,i})-r(x))\Delta_{n,h,i}(x) | \cF_{n,i-1} } \right\}
	\]
	and apply a decomposition as in \eqref{E:UniformExpIneqInnovationsHilbert1}. Set $m= (3+\tilde{\delta}) a_1^{-1} \log n$ and $t=c (n\phi(b))^{-1/2} \log n$, where $c$ is characterized below and $\tilde\delta >0$ sufficiently large. Then
	\begin{align*}
			\p\left(  \sum_{k>m}\big|	\sum_{i=1}^n Z_{e_k,n,i}\big|^2  > t^2 \right) &\le 2 (t^{2}  \phi(h)^2)^{-1}  \sum_{k>m} \E{ (r_{e_k}(X_{n,i})-r_{e_k}(x) )^2 \Delta_{n,h,i}(x)^2 } \\
			&\le \frac{C n \phi(b)}{\phi(h) (\log n)^2 } \left( \sum_{k>m} |\psi'_{2,k,x}(0)| h + L_{2,k} h ^{1+\alpha} \right) \frac{ \E{\Delta_{n,h,i}(x) } }{\phi(h)} \\
			&= \cO\left( \frac{h n \phi(b)}{\phi(h) (\log n)^2} e^{-a_1 m} \right) = \cO\left( \frac{h \phi(b) }{n \phi(h) (\log n)^2} n^{-(1+\tilde{\delta})} \right),
	\end{align*}
	using the definition of $m$ and the assumptions on the coefficients $L_{2,k}$ and the derivatives $\psi'_{2,k,x}$ from \ref{C:Regressor}~\ref{C:Regressor2}. Thus, $\sum_{k>m}\big(	\sum_{i=1}^n Z_{e_k,n,i}\big) e_k = \Oas((n\phi(b))^{-1/2} \log n) = \oas( (n\phi(h))^{-1/2} )$.

	Second, consider the finite dimensional term $\sum_{k\le m} (\sum_{i=1}^n Z_{e_k,n,i} ) \, e_k$. This time, let $t= (n\phi(h))^{-1/2} (\log n)^{-\alpha/4}$. We can derive quite similarly as in \eqref{Eq:LL3_1} and \eqref{Eq:LL3_2} (but this time using \ref{C:Regressor} \ref{C:Regressor2}) that
	\begin{align*}
					\p\left(	\sum_{k \le m} \Big|	\sum_{i=1}^n Z_{e_k,n,i}	\Big|^2  	> t^2 \right) &\le 2 m \exp\left(	- \frac{1}{2} \frac{t^2 m^{-1} n \phi(h) }{Ah + B t m^{-1/2}  }	\right) \\
					&=\cO\left(	\exp\left( - C \frac{(n\phi(h))^{1/2} }{(\log n)^{2+\alpha} } (\log n)^{1+\alpha/2} \right)	\log n \right),
	\end{align*}
	for some $C>0$.	Using that $(n\phi(h))^{1/2} (\log n)^{-(2+\alpha)} \rightarrow \infty$, we find $\sum_{n\in\N} \p\left(	\sum_{k \le m} \big(	\sum_{i=1}^n Z_{e_k,n,i}\big)^2  	> t^2 \right) < \infty$. Thus, $t^{-2} \sum_{k \le m} \big(	\sum_{i=1}^n Z_{e_k,n,i}\big)^2 \rightarrow 0$ $a.c.$ 
	
	This implies that
	$	\sum_{k \le m} \big(	\sum_{i=1}^n Z_{e_k,n,i}\big) e_k$ is $\Oas((n\phi(h))^{-1/2} (\log n)^{-\alpha/4} )$ which is $\oas( (n\phi(h))^{-1/2} )$.
	\end{proof}

\begin{lemma}\label{L:ConvergenceBias0}
	Assume \ref{C:Kernel} -- \ref{C:Covering}. Then
\begin{align}\label{Eq:ConvergenceBias0}
\sup_{y\in U(x,h) }		\norm{ \hat{r}_{n,b}(y) - \hat{r}_{n,b}(x) - (r(y)-r(x)) }  = \oas((n\phi(h))^{-1/2}).
\end{align}
\end{lemma}
\begin{proof}
The terms inside the norm in \eqref{Eq:ConvergenceBias0} are equal to
\begin{align}
			 \frac{\sum_{j=1}^n (Y_{n,j}-r(y)) \Delta_{n,b,j}(y) } {\sum_{j=1}^n \Delta_{n,b,j}(y) } - \frac{\sum_{j=1}^n (Y_{n,j}-r(x))\Delta_{n,b,j}(x) } {\sum_{j=1}^n \Delta_{n,b,j}(x)  }. \label{Eq:ConvergenceBias2}
			\end{align}
We can replace each denominator uniformly in $y$ with the corresponding conditional version because
\begin{align*}			
		(n\phi(b))^{-1}  \sum_{j=1}^n \Delta_{n,b,j}(y)  &= \Big\{ (n\phi(b))^{-1}  \sum_{j=1}^n \Delta_{n,b,j}(y) - \Delta_{n,b,j}(x) - \E{\Delta_{n,b,j}(y) - \Delta_{n,b,j}(x) | \cF_{n,j-1} } \Big\} \\
					&\quad + \Big\{ (n\phi(b))^{-1}  \sum_{j=1}^n \Delta_{n,b,j}(x)  - \E{ \Delta_{n,b,j}(x) | \cF_{n,j-1} } \Big\} \\
					&\quad+  (n\phi(b))^{-1}  \sum_{j=1}^n \E{ \Delta_{n,b,j}(y) | \cF_{n,j-1} } \\
					&= (n\phi(b))^{-1}  \sum_{j=1}^n \E{ \Delta_{n,b,j}(y) | \cF_{n,j-1} } + \oas( (n\phi(h))^{-1/2} ),
\end{align*}			
according to Lemmas~\ref{L:ConvergenceR1} and \ref{L:UnifExpIneqInnovations}. Thus, \eqref{Eq:ConvergenceBias2} can be rewritten as (modulo a remainder which is $\oas((n\phi(h))^{-1/2})$)
			\begin{align}
			&\frac{ (n\phi(b))^{-1} \sum_{j=1}^n (Y_{n,j}-r(y)) \Delta_{n,b,j}(y) - (Y_{n,j}-r(x)) \Delta_{n,b,j}(x) } {(n\phi(b))^{-1} \sum_{j=1}^n \E{ \Delta_{n,b,j}(y) |\cF_{n,j-1}} } \label{Eq:ConvergenceBias2a} \\
			&\quad + \frac{(n\phi(b))^{-1} \sum_{j=1}^n (Y_{n,j}-r(x)) \Delta_{n,b,j}(x)}{(n\phi(b))^{-1} \sum_{j=1}^n \E{\Delta_{n,b,j}(x)  |\cF_{n,j-1} } } \frac{ (n\phi(b))^{-1} \sum_{j=1}^n \E{ \Delta_{n,b,j}(x) - \Delta_{n,b,j}(y) |\cF_{n,j-1} }}{(n\phi(b))^{-1} \sum_{j=1}^n \E{ \Delta_{n,b,j}(y) |\cF_{n,j-1} } }. \label{Eq:ConvergenceBias2b}
	\end{align}
	We begin with \eqref{Eq:ConvergenceBias2b}. The first factor is $\Oas(b) + \oas( (n\phi(h))^{-1/2})$, this follows from Corollary~\ref{Cor:ConvergenceR2} and Lemma~\ref{L:LL3}. We show that the second factor in \eqref{Eq:ConvergenceBias2b} is $\oas(b^\alpha)$ for $\alpha>0$ from \ref{C:SmallBall}. We use the standard expansion to obtain
	\begin{align*}
				|\E{ \Delta_{n,b,j}(x) | \cF_{n,j-1} } - \E{ \Delta_{n,b,j}(y) | \cF_{n,j-1} }| &\le K(1) | F_x^{\cF_{n,j-1}} (b) - F_y^{\cF_{n,j-1}} (b) | \\
				&\quad + \int_0^1 |K'(s)| | F_x^{\cF_{n,j-1}} (bs) - F_y^{\cF_{n,j-1}} (bs) | \intd{s}.  
	\end{align*}
Using the assumptions on the H{\"o}lder continuity of the map $(F^{\cF_{n,j-1}}_x(u)) ^{-1} |F^{\cF_{n,j-1}}_y(u) - F^{\cF_{n,j-1}}_x(u)|$ from \ref{C:SmallBall}, we can consequently derive that (uniformly in $y\in U(x,h)$)
\begin{align*}
  (n\phi(b))^{-1} \sum_{j=1}^n | \E{ \Delta_{n,b,j}(x) - \Delta_{n,b,j}(y) | \cF_{n,j-1} }| &\le 2\big(K(1)+\max_{s\in [0,1]} |K'(s)| \big) h^\alpha\, n^{-1} \sum_{i=1}^n \tilde{L}_{n,i}^2 = \Oas(h^\alpha).
\end{align*}
This finishes the computations for \eqref{Eq:ConvergenceBias2b}. It remains to consider the numerator in \eqref{Eq:ConvergenceBias2a}, i.e.,
	\begin{align}
			& \sup_{y\in U(x,h) } \norm{ (n\phi(b))^{-1} \sum_{i=1}^n (Y_{n,i}-r(y))\Delta_{n,b,i}(y) - (Y_{n,i}-r(x))\Delta_{n,b,i}(x) } \nonumber \\
			\begin{split}\label{Eq:ConvergenceBias3}
			&\le  \sup_{y\in U(x,h) } \Bigg|\Bigg| (n\phi(b))^{-1} \sum_{i=1}^n \Big\{ (Y_{n,i}-r(y))\Delta_{n,b,i}(y) - (Y_{n,i}-r(x))\Delta_{n,b,i}(x) \\
			&\qquad\qquad\qquad\qquad - \E{ (Y_{n,i}-r(y))\Delta_{n,b,i}(y) - (Y_{n,i}-r(x))\Delta_{n,b,i}(x)   | \cF_{n,i-1} } \Big\}  \Bigg|\Bigg|
			\end{split} \\
			&\quad +  \sup_{y\in U(x,h) } \norm{ (n\phi(b))^{-1} \sum_{i=1}^n \E{ (Y_{n,i}-r(y))\Delta_{n,b,i}(y) - (Y_{n,i}-r(x))\Delta_{n,b,i}(x)   | \cF_{n,i-1} } }. \label{Eq:ConvergenceBias3b}
	\end{align}
	First, we treat the summand in \eqref{Eq:ConvergenceBias3b}, we need the following result which follows from \ref{C:Response}
	\begin{align}
				&(n\phi(b))^{-1} \sum_{i=1}^n \E{ (r(X_{n,i}) - r(y))\Delta_{n,b,i}(y) | \cF_{n,i-1} } \nonumber \\
				&= \left\{ (n\phi(b))^{-1} \sum_{i=1}^n \E{ \Delta_{n,b,i}(y) \frac{d(y,X_{n,i})}{b} \Big| \cF_{n,i-1} } \right\} \left\{		\sum_{k\in\N} \psi'_{k,y}(0) e_k \right\} b + \Oas( (b+h)^{1+\alpha} ) \label{Eq:ConvergenceBias4}.
	\end{align}
The first factor in the first term in \eqref{Eq:ConvergenceBias4} is continuous in $x$, more precisely, we infer from Lemma~\ref{L:LL1} that
	\begin{align*}
			\sup_{y\in U(x,h)} (n\phi(b))^{-1} \left| \sum_{i=1}^n \E{ \Delta_{n,b,i}(y) \frac{d(y,X_{n,i})}{b} - \Delta_{n,b,i}(x) \frac{d(x,X_{n,i})}{b}  \Big| \cF_{n,i-1} } \right| =  \Oas(h^\alpha).
	\end{align*}
	Moreover,  
	\[
				\sup_{y\in U(x,h)} \norm{ \sum_{k\in\N} \left( \psi'_{k,y}(0) - \psi'_{k,x}(0) \right) e_k } \le \Big( \sum_{k\in\N} L_k^2  \Big)^{1/2} h^{\alpha} = o( b^{\alpha} ),
	\]
	by assumption. Consequently, using \eqref{Eq:ConvergenceBias4} together with the last two estimates, we obtain that \eqref{Eq:ConvergenceBias3b} is $\Oas(b^{1+\alpha})=\oas((n\phi(h))^{-1/2})$. This finishes the calculations on the summand in \eqref{Eq:ConvergenceBias3b}.

	Second, we split the summand in \eqref{Eq:ConvergenceBias3} in the three summands
	\begin{align}
				&\sup_{y\in U(x,h) } (n\phi(b))^{-1} \norm{ \sum_{i=1}^n Y_{n,i} ( \Delta_{n,b,i}(y) - \Delta_{n,b,i}(x)) - \E{Y_{n,i} ( \Delta_{n,b,i}(y) - \Delta_{n,b,i}(x))|\cF_{n,i-1}	} }, 		\label{Eq:ConvergenceBias5} \\
				&\sup_{y\in U(x,h) }  (n\phi(b))^{-1} \norm{ \sum_{i=1}^n [-r(y)] \big\{  \Delta_{n,b,i}(y) - \Delta_{n,b,i}(x) - \E{  \Delta_{n,b,i}(y) - \Delta_{n,b,i}(x)|\cF_{n,i-1}	} \big\}  }, \label{Eq:ConvergenceBias6} \\
				&\sup_{y\in U(x,h) } (n\phi(b))^{-1} \norm{ \sum_{i=1}^n [- (r(y)-r(x))] \big\{ \Delta_{n,b,i}(x) - \E{ \Delta_{n,b,i}(x))|\cF_{n,i-1}	} \big\} }. \label{Eq:ConvergenceBias7}
	\end{align}
	\eqref{Eq:ConvergenceBias5} is $\oas( (n\phi(h))^{-1/2})$, see Lemma~\ref{L:UniformExpIneqInnovationsHilbert}. Using the fact that $r(y)$ is uniformly bounded above in a neighborhood of $x$ and using Lemma~\ref{L:UnifExpIneqInnovations}, \eqref{Eq:ConvergenceBias6} is also $\oas( (n\phi(h))^{-1/2})$. It remains \eqref{Eq:ConvergenceBias7} which is bounded above by
	\[
			\sup_{y\in U(x,h)} \norm{ r(y) -r(x)} \cdot (n\phi(b))^{-1} \left| \sum_{i=1} ^n \Delta_{n,b,i}(x) - \E{\Delta_{n,b,i}(x)|\cF_{n,i-1}}  \right|.
	\]   
	The supremum in the first factor is $\cO(1)$, the second factor is $\Oas( (n\phi(b))^{-1/2} (\log n)^{1/2}) = \oas( (n\phi(h))^{-1/2})$, see Lemma~\ref{L:ConvergenceR1}. This completes the proof.
	\end{proof}

	\begin{lemma}\label{L:ConvergenceBias}
	Assume \ref{C:Kernel} -- \ref{C:Covering}. Then
	$
				B^*_n(x) - B_n(x) = \oas( (n\phi(h)^{-1/2} ) \text{ in } \cH.
	$
	\end{lemma}
	\begin{proof}
	The bias terms are given in \eqref{Eq:CondBiasBoot} and \eqref{Eq:DefCondBias}. We use the decomposition
	\begin{align}
			B_{n}^*(x) &= \frac{ (n\phi(h))^{-1} \sum_{i=1}^n (\hat{r}_{n,b}(X_{n,i})- \hat{r}_{n,b}(x) )\Delta_{n,h,i}(x) }{ (n\phi(h))^{-1} \sum_{i=1}^n \Delta_{n,h,i}(x) }, \label{Eq:ConvergenceBias0a} \\
			B_n(x) &= \frac{ (n\phi(h))^{-1} \sum_{i=1}^n \E{ (r(X_{n,i})- r(x) ) \Delta_{n,h,i}(x) | \cF_{n,i-1} } }{ (n\phi(h))^{-1} \sum_{i=1}^n \E{\Delta_{n,h,i}(x) | \cF_{n,i-1} } }. \label{Eq:ConvergenceBias0b}
	\end{align}
	We infer from Lemma~\ref{L:ConvergenceR1} that 
	$
				(n\phi(h))^{-1} |\sum_{i=1}^n \Delta_{n,h,i}(x) - \mathbb{E}[\Delta_{n,h,i}(x) | \cF_{n,i-1}] | = \Oas( (n \phi(h))^{-1/2} (\log n)^{1/2} )
	$
and $f_1(x)>0$ by assumption. Moreover, the numerator of \eqref{Eq:ConvergenceBias0b} is $\Oas(h)$, see the proof of \ref{L:LL3}. This means that we can exchange the denominator in \eqref{Eq:ConvergenceBias0b} with the denominator from \eqref{Eq:ConvergenceBias0a} at $\Oas(h (n\phi(h))^{-1/2} (\log n)^{1/2} ) =  \oas( (n\phi(h))^{-1/2} )$ costs. Consequently, it suffices to consider the difference
	\begin{align}
				&(n\phi(h))^{-1} \norm{ \sum_{i=1}^n (\hat{r}_{n,b}(X_{n,i})- \hat{r}_{n,b}(x) )\Delta_{n,h,i}(x) - \E{ (r(X_{n,i})- r(x) ) \Delta_{n,h,i}(x) | \cF_{n,i-1} }  } \nonumber \\
				& \le (n\phi(h))^{-1} \norm{ \sum_{i=1}^n \{ \hat{r}_{n,b}(X_{n,i}) - \hat{r}_{n,b}(x) - (r(X_{n,i})-r(x)) \} \Delta_{n,h,i}(x) } \label{Eq:ConvergenceBias1} \\
				&\quad+ (n\phi(h))^{-1} \norm{ \sum_{i=1}^n ( r(X_{n,i})-r(x))\Delta_{n,h,i}(x) - \E{ (r(X_{n,i})- r(x) ) \Delta_{n,h,i}(x) | \cF_{n,i-1} } } . \label{Eq:ConvergenceBias1b} 
	\end{align}
	The term inside the curly parentheses in \eqref{Eq:ConvergenceBias1} is bounded above by
	$
				\sup_{y\in U(x,h)} \norm{ \hat{r}_{n,b}(y) - \hat{r}_{n,b}(x) - (r(y)-r(x)) }
	$
	which is $\oas( (n\phi(h))^{-1/2}$ by \eqref{Eq:ConvergenceBias0}. \eqref{Eq:ConvergenceBias1b} also attains the desired rate, this follows directly from Lemma~\ref{L:ConvergenceBiasIngredients}. 
	\end{proof}

\begin{lemma}\label{L:DistributionBS}
	Let assumptions \ref{C:Kernel} -- \ref{C:Covering} be satisfied. Consider the wild bootstrap procedure. Then
\[
	\cL^*\left( \sqrt{n\phi(h)}  \frac{ \sum_{i=1}^n \left(	Y^*_{n,i} - \Ec{Y^*_{n,i}| \cF^*_{n,i-1} } \right) \Delta_{n,h,i}(x) }{\sum_{i=1}^n \Delta_{n,h,i}(x) } \right) \overset{d}{\rightarrow} \fG(0,\cC_x) \quad a.s.
\]
Moreover, assume that additionally \ref{C:NaiveBS} is satisfied. Then the statement is also true for the naive bootstrap procedure.
\end{lemma}
\begin{proof}
The quotient can be rewritten as
	\begin{align}\label{E:GaussianLimitBootstrap4}
				 \frac{  (n\phi(h))^{-1/2} \sum_{i=1}^n \left(	Y^*_{n,i} - \Ec{Y^*_{n,i}| \cF^*_{n,i-1} } \right) \Delta_{n,h,i}(x) }{(n\phi(h))^{-1} \sum_{i=1}^n \Delta_{n,h,i}(x) } = \frac{  (n\phi(h))^{-1/2} \sum_{i=1}^n \epsilon^*_{n,i}   \Delta_{n,h,i}(x) }{(n\phi(h))^{-1} \sum_{i=1}^n \Delta_{n,h,i}(x) }.
	\end{align}
	The denominator converges to $M_1 f_1(x)>0$ $a.s.$ Thus, in order to prove the asymptotic distribution, we need to verify the Lindeberg condition~\ref{C:Lindeberg} \ref{C:Lindeberg1} to \ref{C:Lindeberg3} in the bootstrap world for the numerator. Define
	$
			\xi_{n,i}^* \coloneqq (n \phi(h))^{-1/2} \epsilon^*_{n,i} \Delta_{n,h,i}(x)
	$
	for $i=1,\ldots,n$ and $n\in\N$.
	
	We start with the wild bootstrap. Let $v=\sum_{k\in\N} \gamma_k e_k \in\cH$ be arbitrary but fixed. Then
	\begin{align}\label{E:GaussianLimitBootstrap5}
				\sum_{i=1}^n \Ec{\left(\xi_{v,n,i}^*\right)^2 |\cF_{n,i-1}} = \sum_{i=1}^n\frac{ \hat\epsilon_{v,n,i}^2 \Delta_{n,h,i}^2(x)}{n \phi(h) } = \sum_{i=1}^n \frac{(\epsilon_{v,n,i}+r_v(X_{n,i}) - \hat{r}_{v,n,b}(X_{n,i}))^2 \Delta_{n,h,i}^2(x)}{n \phi(h) }.
	\end{align}
	The difference $|r_v(X_{n,i}) - \hat{r}_{v,n,b}(X_{n,i})|$ is bounded above by
	\begin{align}
	\begin{split}\label{E:GaussianLimitBootstrap6}
				\sup_{y\in U(x,h) } |\hat{r}_{v,n,b}(y) - r_v(y)| &\le |\hat{r}_{v,n,b}(x) - r_v(x)| \\
				&\quad+\sup_{y\in U(x,h) } |\hat{r}_{v,n,b}(y) - r_v(y) - [\hat{r}_{v,n,b}(x) - r_v(x)] | .
				\end{split}
	\end{align}
	The first term in \eqref{E:GaussianLimitBootstrap6} is of order $\Oas(b+(n\phi(b))^{-1/2} (\log n)^{1/2} )$ by Lemma~\ref{L:ConvergenceR1} and Lemma~\ref{L:LL3}. Thus, it remains to show that the second term in \eqref{E:GaussianLimitBootstrap6} vanishes also $a.s.$ This follows however from Lemma~\ref{L:ConvergenceBias0}.

	Using this insight, we see that \eqref{E:GaussianLimitBootstrap5} equals
	\begin{align*}
				&(n\phi(h))^{-1}\sum_{i=1}^n (\epsilon_{v,n,i} + \oas(1) )^2 \Delta_{n,h,i}^2(x) = (n\phi(h))^{-1} \sum_{i=1}^n (\epsilon_{v,n,i}  )^2 \Delta_{n,h,i}^2(x)+ \oas(1) \\
				&= \E{\epsilon_{v,n,i}^2|X_{n,i}=x} M_2 f_1 (x) + \oas(1) =\Big(\sum_{j,k\in\N} \gamma_j\gamma_k W_{2,j,k}(x) \Big) M_2 f_1(x) + \oas(1).
	\end{align*}
	This shows that Condition~\ref{C:Lindeberg}~\ref{C:Lindeberg1} is satisfied and in particular that the conditional variance of \eqref{E:GaussianLimitBootstrap4} given the sample $\cS_n$ and in direction $v$ converges to
	\[
		\operatorname{Var}^* \left( \frac{  (n\phi(h))^{-1/2} \sum_{i=1}^n \epsilon^*_{v,n,i}   \Delta_{n,h,i}(x) }{(n\phi(h))^{-1} \sum_{i=1}^n \Delta_{n,h,i}(x) } \right) \rightarrow \frac{ \E{\epsilon_{v,n,i}^2|X_{n,i}=x} M_2}{M_1^2 f_1 (x) } \quad a.s.
		\]		
	Moreover, Condition~\ref{C:Lindeberg}~\ref{C:Lindeberg2} is satisfied. Indeed, use that $\Ec{\left(\xi_{v,n,i}^*\right)^2 } = \Ec{\left(\xi_{v,n,i}^*\right)^2 |\cF^*_{n,i-1}}$ $a.s.$ Then
	\begin{align}
		&\left| \sum_{k\in\N} \sum_{i=1}^n \Ec{\left(\xi_{e_k,n,i}^*\right)^2 |\cF^*_{n,i-1}} - M_2 f_1 (x) \sum_{k\in\N}  W_{2,k,k}(x) \right| \nonumber \\
		\begin{split}\label{E:GaussianLimitBootstrap6b}
		&\le (n\phi(h))^{-1} \left| \sum_{i=1}^n \norm{\epsilon_{n,i}}^2  \Delta_{n,h,i}(x)^2 - \E{  \norm{\epsilon_{n,i}}^2 \Delta_{n,h,i}(x)^2 | \cF_{n,i-1} } \right| \\
		&\quad+ (n\phi(h))^{-1} \left| \sum_{k\in\N} \sum_{i=1}^n  \E{  \epsilon_{e_k,n,i}^2 \Delta_{n,h,i}(x)^2 | \cF_{n,i-1} } - M_2 f_1 (x) \sum_{k\in\N}  W_{2,k,k}(x)\right| + \oas(1).
		\end{split}
	\end{align}
	Using the assumptions on the conditional moments of $\norm{Y_{n,i}}$, we can derive an exponential inequality for the first line in \eqref{E:GaussianLimitBootstrap6b} as in Lemma~\ref{L:ConvergenceR1}. We then find that the this summand is $\Oas( (n\phi(h))^{-1/2} (\log n)^{1/2} )$. Moreover, the second line in \eqref{E:GaussianLimitBootstrap6b} is $\oas(1)$, this follows with arguments similar to those used in \eqref{Eq:NormalityQ2b}. Hence,  Condition~\ref{C:Lindeberg}~\ref{C:Lindeberg2} is satisfied.

	Finally, we show Condition~\ref{C:Lindeberg}~\ref{C:Lindeberg3}. Let $\rho>0$, $a=1+\delta'/2$ and $b$ H{\"o}lder conjugate to $a$, where $\delta'$ is from \ref{C:Response}~\ref{C:Response3}.
	\begin{align}
			\sum_{i=1}^n \Ec{ (\xi^*_{e_k,n,i})^2 \1{|\xi^*_{e_k,n,i}| > \rho } | \cF^*_{n,i-1} } & \le \rho^{-2a/b} \sum_{i=1}^n  \Ec{|\xi^*_{e_k,n,i}|^{2a} | \cF^*_{n,i-1} } \nonumber  \\
			&=  \rho^{-2a/b}  \sum_{i=1}^n  \frac{\Delta^{2+\delta'}_{n,h,i}(x) }{ (n\phi(h))^{1+\delta'/2} } \, \Ec{|\epsilon^*_{e_k,n,i}|^{2+\delta'} | \cF^*_{n,i-1} } \nonumber \\
			&=  (n\phi(h))^{-\delta'/2} \frac{ \E{ |V_{n,i}|^{2+\delta'} } }{\rho^{2a/b} }  \frac{ \sum_{i=1}^n |\hat\epsilon_{e_k,n,i}|^{2+\delta'} \, \Delta^{2+\delta'}_{n,h,i}(x) }{ n\phi(h) }, \label{E:GaussianLimitBootstrap7}
	\end{align}
	where we use that $\epsilon^*_{n,i} = V_{n,i} \hat\epsilon_{n,i}$ and $\E{ |V_{n,i}|^{2+\delta'}}<\infty$. Next, use that 
	\[
		|\hat\epsilon_{e_k,n,i}| \le \norm{Y_{n,i}} + \frac{\sup_{y\in U(x,b)} (n\phi(b))^{-1} \sum_{i=1}^n \norm{Y_{n,i}} \Delta_{n,b,i}(y) }{ \inf_{y\in U(x,b)} (n\phi(b))^{-1} \sum_{i=1}^n \Delta_{n,b,i}(y) } = \norm{Y_{n,i}} + \Oas(1).
	\]
		Next, use \ref{C:Response}~\ref{C:Response3} and proceed similar as in the proof of Lemma~\ref{L:ConvergenceR1} to deduce that the last factor in \eqref{E:GaussianLimitBootstrap7} is $a.s.$ bounded above. Hence, \eqref{E:GaussianLimitBootstrap7} is $\Oas(  (n\phi(h))^{-\delta'/2} )$. This proves the statement for the wild bootstrap.

	It follow the considerations for the naive bootstrap of \eqref{E:GaussianLimitBootstrap4} under the assumptions that the $\hat{\epsilon}_{n,i}$ satisfy \ref{C:NaiveBS}.
	Note that $n^{-1}\sum_{i=1}^n \hat{\epsilon}_{v,n,i}^2 \rightarrow \sum_{j,k \in \N}\gamma_j\gamma_k W_{2,j,k} = \E{\epsilon_{v,n,i}^2}$ $a.s.$ for all $v=\sum_{k\in\N} \gamma_k e_k \in \cH$ using Lemma~\ref{L:NaiveBS}. Condition~\ref{C:Lindeberg}~\ref{C:Lindeberg1} holds also in this case, we have
	\begin{align*}
				\sum_{i=1}^n \Ec{\left(\xi_{v,n,i}^*\right)^2 |\cF_{n,i-1}} = \left( \E{\epsilon_{v,n,i}^2} + \oas(1)		\right) \frac{\sum_{i=1}^n  \Delta_{n,h,i}^2(x)  }{n \phi(h) } \rightarrow \E{\epsilon_{v,n,i}^2} M_2 f_1(x) \quad a.s.
	\end{align*}
	Condition~\ref{C:Lindeberg}~\ref{C:Lindeberg2} is also satisfied. Indeed,
	\begin{align*}
		\left| \sum_{k\in\N} \sum_{i=1}^n \Ec{\left(\xi_{v,n,i}^*\right)^2 |\cF^*_{n,i-1}} - M_2 f_1 (x) \sum_{k\in\N}  W_{2,k,k} \right|  &= M_2 f_1(x) \Big|  \sum_{k\in\N} \Big(n^{-1} \sum_{j=1}^n \hat{\epsilon}_{e_k,n,j}^2 \Big) - W_{2,k,k} \Big| + \oas(1) \\
		&= M_2 f_1(x) \Big| n^{-1} \sum_{j=1}^n \norm{\hat{\epsilon}_{n,j}}^2 - \E{\norm{\epsilon_{n,j} }^2} \Big| + \oas(1).
	\end{align*}
	By assumption, this last equality vanishes $a.s.$ This establishes (ii). Furthermore, Condition~\ref{C:Lindeberg}~\ref{C:Lindeberg3} is true arguing along the same lines as in the derivation of \eqref{E:GaussianLimitBootstrap7} and using that $n^{-1} \sum_{i=1}^n |\hat{\epsilon}_{v,n,i}|^{2+\delta'}$ is $\Oas(1)$.
	\end{proof}

	\begin{proof}[Proof of Theorem~\ref{Thrm:GaussianLimitBootstrap}]
	So far, we have $\cL(\sqrt{n\phi(h)} ( \hat{r}_{n,h}(x) - r(x)	 - B_n(x) ) ) \Rightarrow \fG(0,\cC_x)$ and $\cL^*(\sqrt{n\phi(h)} ( \bootr(x) - \hat{r}_{n,b}(x)	 - B^*_n(x) ) )\Rightarrow \fG(0,\cC_x)$ $a.s.$ Consider the bootstrapped operator and use the following decomposition
	\begin{align}	\begin{split}\label{E:GaussianLimitBootstrap1}
					\sqrt{n\phi(h)} \left( \bootr(x) - \hat{r}_{n,b}(x)		\right) &=  \sqrt{n\phi(h)}  \frac{ \sum_{i=1}^n \left(	Y^*_{n,i} - \Ec{Y^*_{n,i}| \cF^*_{n,i-1} } \right) \Delta_{n,h,i}(x) }{\sum_{i=1}^n \Delta_{n,h,i}(x) } \\
					&\quad + \sqrt{n\phi(h)} B^*_n(x) =: A^*_n + b^*_n.
	\end{split}\end{align}
	The last term in \eqref{E:GaussianLimitBootstrap1} equals $b^*_n = \sqrt{n\phi(h)} B^*_n(x) =  \sqrt{n\phi(h)} B_n(x) + \oas(1)$ from Lemma~\ref{L:ConvergenceBias}. Moreover,
	\begin{align}
			\sqrt{n\phi(h)} ( \bootr(x) - \hat{r}_{n,b}(x)	) &= A_n + b_n 
	\end{align}
	where $A_n = \sqrt{n\phi(h)}(Q_n(x)+R_n(x))/\hat{r}_{n,h,1}(x)$ and $b_n= \sqrt{n\phi(h)} B_n(x)$ from \eqref{Eq:GaussianLimit}. Then there is a deterministic and bounded sequence $(c_n:n\in\N)\subseteq \cH$ such that $\norm{b_n - c_n} = \oas(1) $ and $\norm{b_n-b^*_n} = \oas(1)$.
	
	Let $\Psi\colon\cH\to\R$ be Lipschitz-continuous and bounded. Use the definitions of $\mu_{n,x}$ and $\mu^*_{n,x}$ from \eqref{DefMeasures}. Then
	\begin{align*}
				\left| \int_{\cH} \Psi \intd{\mu^*_{n,x}} -   \int_{\cH} \Psi \intd{\mu_{n,x}}\right| &\le \left| \Ec{ \Psi(A^*_n + b^*_n) } -  \Ec{ \Psi(A^*_n + c_n) } \right| \\
				&\quad + \left| \Ec{ \Psi(A^*_n + c_n) } - \int_\cH \Psi \intd{\fG(c_n,\cC_x)} \right| \\
				&\quad + \left| \int_\cH \Psi \intd{\fG(c_n,\cC_x)} - \E{ \Psi(A_n + c_n) } \right| \\
				&\quad + \left|\E{ \Psi(A_n + c_n) } - \E{\Psi(A_n + b_n)} \right|.
	\end{align*}

The first and the last terms on the right-hand side converge to 0. Denote the Lipschitz constant of $\Psi$ by $L_{\Psi}$. We obtain for the first term $| \Ec{ \Psi(A^*_n + b^*_n) } -  \Ec{ \Psi(A^*_n + c_n) } |\le \Ec{ ( L_\Psi  \norm{ b^*_n - c_n} ) \wedge \norm{\Psi}_\infty }$. This converges to 0 by Lebesgue's dominated convergence theorem if we use the two-stage probabilistic nature of the bootstrap, which means that once the sample data $\cS_n$ are realized, the resampling scheme is performed independently of $\cS_n$.

Consider the second and the third term on the right-hand side, again, we only give the details for the second term. We have that $\cL^*(A_n^*) \Rightarrow \fG(0,\cC_x)$ $a.s.$ by Lemma~\ref{L:DistributionBS}, thus, by Theorem 3.1 in \cite{rao1962relations}
\[
		\sup_{\Xi \in\cG} | \Ec{ \Xi(A_n^*) } - \int_\cH \Xi \intd{ \fG(0,\cC_x)} | \rightarrow 0 \quad a.s., \quad (n\rightarrow\infty), 
\]
where $\cG$ is any class of real-valued, uniformly bounded and equicontinuous functionals on $\cH$. In particular, the class $\cG^* = \{ \Psi(\cdot - c_n): n\in \N\}$ satisfies this assumption. Hence,
\[
			\left| \Ec{ \Psi(A^*_n + c_n) } - \int_\cH \Psi \intd{\fG(c_n,\cC_x)} \right|  \le \sup_{\Xi\in\cG^*} 	\left| \Ec{ \Xi(A^*_n) } - \int_\cH \Xi \intd{\fG(0,\cC_x)} \right| \rightarrow 0 \quad a.s., \quad (n\rightarrow\infty).
\]
\end{proof}

\begin{proof}[Proof of Corollary~\ref{CorBS}]
Clearly, in this case $\mu_{n,x} \Rightarrow \fG(\bar{B}(x),\cC_x )$ and $\mu_{n,x}^* \Rightarrow \fG(\bar{B}(x),\cC_x )$ $a.s.$ by Slutzky's theorem and Theorem~\ref{Thrm:GaussianLimitBootstrap} and Lemma~\ref{L:LL3}. The result concerning the one-dimensional projections is then a consequence of Polya's theorem, see, e.g. \cite{serfling2009approximation}.
\end{proof}

	\appendix
	\section{Appendix}
	
	\begin{lemma}[Details on Example~\ref{Ex:SAR}]\label{Pf:ExSAR}\
	\begin{itemize}\setlength\itemsep{0em}

	\item [\mylabel{ExA}{(a)}]	Let $\theta_{\ell,1-\ell}$ be a linear operator on $\cH$ and assume that $\norm{\theta_{\ell,1-\ell}}_{\cL(\cH)} < 1/2$ for $\ell \in\{0,1\}$. Then the SAR process from \eqref{Eq:SAR1} admits the stationary solution from \eqref{Eq:SAR2} and the series in \eqref{Eq:SAR2} converges in $L^2_\cH(\p)$ and $a.s.$
	\item [\mylabel{ExB}{(b)}] Let $\theta_{\ell,1-\ell}$ be Lipschitz-continuous w.r.t. $\|\cdot\|$ with Lipschitz constant smaller than $1/2$ for $\ell\in\{0,1\}$. Let the $\epsilon_{(i,j)}$ take values in $\cE'\subseteq C^1_{M'}$ for some $M'\in\R_+$. Then the $X_{(i,j)}$ take values in $C^1_M$ for some $M\in\R_+$.	
	\end{itemize}
	\end{lemma}
	\begin{proof}
	We begin with \ref{ExA}. Iterating the definition of $X_{(i,j)}$, one finds for any $p\in\N_+$
		\begin{align*}
				X_{(i,j)} &= \sum_{ k \in \{0,1\}^p } \left(\prod_{\ell=1}^p \theta_{k_\ell,1-k_\ell} \right) ( X_{(i-k.,j-(p-k.))}) \\
				&\quad + \sum_{u=0}^{p-1} \sum_{ k \in \{0,1\}^u } \left(\prod_{\ell=1}^u \theta_{k_\ell,1-k_\ell} \right) (\epsilon_{(i-k.,j-(u-k.))}).
	\end{align*}
Hence, it remains to prove that the limit in \eqref{Eq:SAR2} has the claimed convergence properties. Therefore, observe that
\begin{align*}
		&\E{	\norm{ \sum_{u=0}^\infty \sum_{ k \in \{0,1\}^u } \left(\prod_{\ell=1}^u \theta_{k_\ell,1-k_\ell} \right) (\epsilon_{(i-k.,j-(u-k.))})   }^2	} \\
		&=\E{ \norm{ \sum_{u=0}^\infty \sum_{v=0}^\infty \sum_{\substack{k\in\{0,1\}^{u+v}\\ k.=u} } \left(\prod_{\ell=1}^{u+v} \theta_{k_\ell,1-k_\ell} \right) (\epsilon_{(i-u,j-v)})}^2	  }.
\end{align*}
Set $\sigma^2 = \E{\| \epsilon_{(i,j)} \|^2}$. Using the orthogonality of the $\epsilon_{(i,j)}$, we find as in the proof of Theorem~3.1 in \cite{bosq_linear_2000} that this last line is bounded above by
\begin{align*}
		 \sigma^2 \times \sum_{u=0}^\infty \sum_{v=0}^\infty \binom{u+v}{u}^2 \norm{\theta_{1,0}}_{\cL(\cH)}^{2u} \norm{\theta_{0,1}}_{\cL(\cH)}^{2v} \le  \sigma^2 \times \sum_{t=0}^\infty (2\max\{ \norm{\theta_{1,0}}_{\cL(\cH)}, \norm{\theta_{0,1}}_{\cL(\cH)}\} )^{2t} < \infty.
\end{align*}
Moreover, the series converges $a.s.$ Indeed, 
\begin{align*}
		&\E{	\left( \sum_{u,v=0}^\infty \binom{u+v}{u} \norm{\theta_{1,0}}_{\cL(\cH)}^u \norm{\theta_{1,0}}_{\cL(\cH)}^v \norm{\epsilon_{(i-u,j-v)} }		\right)^2	} \\
		&\le \sigma^2 \times \left( \sum_{t=0}^\infty (2\max\{ \norm{\theta_{1,0}}_{\cL(\cH)}, \norm{\theta_{0,1}}_{\cL(\cH)}\} )^{t} \right)^2 < \infty.
\end{align*}
Thus, $\sum_{u,v=0}^\infty \binom{u+v}{u} \norm{\theta_{1,0}}_{\cL(\cH)}^u \norm{\theta_{1,0}}_{\cL(\cH)}^v \norm{\epsilon_{(i-u,j-v)} }		< \infty$ $a.s.$ In particular, the series on the left-hand-side of \eqref{Eq:SAR2} converges $a.s.$

To prove \ref{ExB}, we make use of the Lipschitz continuity of the operators $\theta_{\ell,1-\ell}$. Denote the Lipschitz constant for both operators by $L_\theta < 1/2$. Then we find
\[
				\norm{ X_{(i,j)} } \le L_\theta \norm{ X_{(i-1,j)}} + \norm{ \theta_{1,0}(0) } + L_\theta \norm{ X_{(i,j-1)}} + \norm{ \theta_{0,1}(0) } +  \norm{ \epsilon_{(i,j )} }.
\]
If we iterate this inequality, we obtain similar as above that
\[
			\norm{ X_{(i,j)} } \le L_\theta^r \sum_{k=0}^r \binom{r}{k} \norm{ X_{(i-k,j-(r-k) )}} + \sum_{u=0}^{r-1} L_\theta^u \sum_{k=0}^u \binom{u}{k} \left( \norm{ \epsilon_{(i-k,j-(u-k) )} } + \norm{ \theta_{1,0}(0) }  + \norm{ \theta_{0,1}(0) } \right)
\]
for all $r\in\N$. Letting $r\rightarrow \infty$, this implies 
\[
	\norm{ X_{(i,j)} } \le (1-2L_\theta)^{-1} (M' |\cD|^{1/2} + \norm{ \theta_{1,0}(0) }  + \norm{ \theta_{0,1}(0) } ),
	\]
because the innovations are uniformly bounded above in the $\norm{\cdot}$-norm by $M' |\cD|^{1/2}<\infty$.
Write $L_{X_{(i,j)}}$ for the Lipschitz constant of $X_{(i,j)}$. The smoothness requirement from \eqref{C:Smoothness} on the operators $\theta_{\ell,1-\ell}$ allows us to deduce
\[
			L_{X_{(i,j)}} \le 2 A (1+\| X_{(i,j)} \| ) + M'
\]
for all $(i,j)$. Consequently, $\max_{t\in\cD} |X_{(i,j)}(t)|$ is also bounded uniformly in $(i,j)$. Hence, there is some $M>0$ such that the $X_{(i,j)}$ take values in $C^1_M$ for all $(i,j)$. This yields the claim.
	\end{proof}

	\begin{lemma}\label{L:NaiveBS}
	Assume \ref{C:NaiveBS} and let $v=\sum_{k\in\N} \gamma_k e_k \in \cH$. Then $n^{-1}\sum_{i=1}^n \hat{\epsilon}_{v,n,i}^2 \rightarrow\sum_{j,k \in \N}\gamma_j\gamma_k W_{2,j,k}$ $a.s.$
	\end{lemma}
	\begin{proof}
	Let $\rho>0$ be arbitrary but fixed. Set
	\[
				\tilde\rho \coloneqq \frac{\rho }{8 \norm{v}}\left \{	\max\left[ \limsup_{n\rightarrow\infty} \left(n^{-1} \sum_{i=1}^n \norm{\hat{\epsilon}_{n,i} }^{2+\delta'} \right)^{2/(2+\delta')}, 	\sum_{k\in\N} W_{2,k,k}	\right]\right\}^{-1} > 0 \quad  a.s.
	\]
	Moreover, set $m = \inf\{ u\in\N: \sum_{k: k>u} \gamma_k^2 \le \tilde\rho^2 \}$, then $m<\infty$ $a.s.$ Set $I_m = \{ (j,k): j>m \text{ or } k>m \}$ which is $a.s.$ non-empty. Consider
	\begin{align}
					&\left| \sum_{j,k \in \N} n^{-1} \sum_{i=1}^n \gamma_j \gamma_k \hat\epsilon_{e_j,n,i} \hat\epsilon_{e_k,n,i} - \sum_{j,k\in\N}  \gamma_j \gamma_k  W_{2,j,k} \right | 	\nonumber			\\
					&\le \left| \sum_{j,k \le m} n^{-1} \sum_{i=1}^n \gamma_j \gamma_k \hat\epsilon_{e_j,n,i} \hat\epsilon_{e_k,n,i} - \sum_{j,k\le m} \gamma_j \gamma_k W_{2,j,k} \right | \label{Eq:NaiveBS1} \\
					&\quad +  \sum_{(j,k) \in I_m } n^{-1} \sum_{i=1}^n \left|\gamma_j \gamma_k \hat\epsilon_{e_j,n,i} \hat\epsilon_{e_k,n,i} \right | +  \sum_{(j,k)\in I_m } \left|\gamma_j \gamma_k W_{2,j,k}  \right |.  \label{Eq:NaiveBS2} 
	\end{align}	
	Using the relation $W_{2,j,k} \le W_{2,j,j}^{1/2} W_{2,k,k}^{1/2}$, an $L^p$-inequality and the Cauchy-Schwarz-inequality, one obtains that both summands in the last line \eqref{Eq:NaiveBS2} are bounded above by
	\begin{align*}
				\sum_{(j,k)\in I_m } \left|\gamma_j \gamma_k W_{2,j,k}  \right | &\le 2 \tilde\rho \norm{v} \sum_{k\in\N} W_{2,k,k} \le \frac{\rho}{4} \quad a.s.,\\
				\sum_{(j,k) \in I_m } n^{-1} \sum_{i=1}^n \left|\gamma_j \gamma_k \hat\epsilon_{e_j,n,i} \hat\epsilon_{e_k,n,i} \right | &\le 2\tilde\rho  \norm{v} n^{-1} \sum_{i=1}^n \norm{\hat\epsilon_{n,i}}^2 \le \frac{\rho}{4} \quad a.s.
	\end{align*}
	Furthermore, the term in \eqref{Eq:NaiveBS1} converges to 0 $a.s.$, using the fact that $n^{-1} \sum_{i=1}^n \hat\epsilon_{e_j,n,i} \hat\epsilon_{e_k,n,i} $ converges to $W_{2,j,k}$ $a.s.$ for each pair $j,k$ and that $m< \infty$ $a.s.$
	\end{proof}


\end{document}